\documentclass[11pt]{amsart}

\usepackage[top=1in, bottom=1in, left=1.3in, right=1.3in]{geometry}
\usepackage{changes}

\usepackage{amsmath}
\usepackage{amssymb}
\usepackage{amsthm}
\usepackage{verbatim}
\usepackage{graphicx}
\usepackage{graphics}
\usepackage{color}
\usepackage{enumerate}
\usepackage{mathrsfs}
\usepackage{booktabs}
\usepackage[toc,page]{appendix}
\usepackage{enumerate}
\usepackage{mathrsfs}
\usepackage{fancyhdr}
\usepackage{amsopn, amstext, amscd,pifont}
\usepackage{amssymb,bm, cite}
\usepackage[all]{xy}
\usepackage{color}
\usepackage{graphicx,hyperref}
\usepackage{float}
\usepackage{listings}
\usepackage{setspace}
\usepackage{url}
\usepackage{paralist}
\usepackage{subcaption}
\usepackage{stmaryrd}
\usepackage{tikz-cd}

\usepackage{graphicx,pgf}
\usepackage{float}
\usepackage{enumitem}

\usepackage{fancyhdr}

\numberwithin{equation}{section} 

\newtheorem{introtheorem}{Theorem}

\theoremstyle{plain}

\newtheorem{theorem}{Theorem}[section]
\newtheorem{proposition}[theorem]{Proposition}
\newtheorem{lemma}[theorem]{Lemma}
\newtheorem{corollary}[theorem]{Corollary}

\theoremstyle{definition}
\newtheorem{definition}[theorem]{Definition}

\newtheorem{remark}[theorem]{Remark}

\graphicspath{{images/}}

  \def\sl2c{\operatorname{PSL(2,\C)}}

\def\C{\mathbb{C}}

\def\H{{\mathbb{H}}}
\def\R{\mathbb{R}}

\def\P{\mathbb{P}}

\def\N{\mathbb{N}}

\def\CP2{{\mathbb{CP}^2}}

\def\bA{\mathbf{A}}
\def\bB{\mathbf{B}}
\def\Ratd{{\operatorname{Rat}_d}}
\def\Rat{{\operatorname{Rat}}}
\def\oRat{\overline{\operatorname{Rat}}}
\def\oRatd{\overline{\operatorname{Rat}}_d}
\def\oRatdn{\overline{\operatorname{Rat}}_{d^n}}
\def\wRatd{\widehat{\Rat}_d}
\def\ratd{{\operatorname{rat}_d}}
\def\rat{{\operatorname{rat}}}
\def\orat{\overline{\operatorname{rat}}}
\def\oratd{\overline{\operatorname{rat}}_d}
\def\oratdn{\overline{\operatorname{rat}}_{d^n}}
\def\wratd{\widehat{\rat}_d}

\def\ratd{{\operatorname{rat}_d}}
\def\rat{{\operatorname{rat}}}

\def\crit{\operatorname{Crit}}

\def\res{\operatorname{Res}}

\def\ex{Ex}
\usepackage{calrsfs}
\DeclareMathAlphabet{\pazocal}{OMS}{zplm}{m}{n}

\def\cA{{\pazocal{A}}}
\def\cB{{\pazocal{B}}}

\def\cE{{\pazocal{E}}}
\def\cF{{\pazocal{F}}}

\def\cJ{{\pazocal{J}}}

\def\cS{{\pazocal{S}}}
\def\cT{{\pazocal{T}}}

\def\oC{\overline{\C}}

\def\tvarphi{\tilde{\varphi}}

\def\tphi{\tilde{\phi}}

\def\hole{\mathrm{Hole}}

\def\so3{{\operatorname{SO(3)}}}

\usepackage[pagewise]{lineno}

\begin{document}

\title{Compactifications and measures for rational maps}
\author{Jan Kiwi and Hongming Nie}
\thanks{The second author acknowledges the support of ANID/FONDECYT Regular
1240508.} 
\address{Facultad de Matem\'aticas, Pontificia Universidad Cat\'olica de Chile} 
\email{jkiwi@uc.cl}
\address{Los Angeles, CA 90089}
  \email{hongming.i.nie@gmail.com}
\maketitle
\begin{abstract}
  We study extensions of the measure of maximal entropy to suitable compactifications of the parameter space and the moduli space of rational maps acting on the Riemann sphere.
For parameter space, we consider a space which resolves the discontinuity of the iterate map. We show 
  that the measure of maximal entropy extends continuously
  to this resolution space. 
  For moduli space, we consider a space which resolves
  the discontinuity of the iterate map acting on its geometric invariant theory compactification. We show that the measure of maximal entropy, barycentered and modulo rotations, also extends continuously this resolution space. Thus, answering in the positive a question raised by DeMarco.
  A main ingredient is a description of limiting dynamics for some sequences. 
\end{abstract}

\tableofcontents

\section{Introduction}
\label{s:introduction}
The aim of this paper is to study iterations of rational maps
in one complex variable near \emph{degeneration}.
The emphasis will be on a further exploration of the interplay
between \emph{iterate maps} and \emph{measures of maximal entropy},
initiated by DeMarco \cite{DeMarco05,DeMarco07}  and
continued by DeMarco and Faber \cite{DeMarco14},
Favre \cite{Favre20}, Favre and Gong \cite{Favre24} and
Okuyama \cite{Okuyama20,Okuyama24}.

The space of rational functions $f \in \C(z)$
of degree $d \ge 2$, denoted by $\Ratd$, is naturally
identified with a Zariski open subset of $\P^{2d+1}$.
This identification furnishes $\Ratd$ with the compactification
$\oRatd \equiv \P^{2d+1}$.
Elements of the codimension one algebraic set $\partial \Ratd:= \oRatd \setminus \Ratd$
are called \emph{degenerate rational maps of degree $d$}. See \S \ref{s:spaces}.

For $n \ge 2$, the \emph{$n$-th iterate map}
$\Phi_n: \Ratd \to \Rat_{d^n}$, sending $f$ to its $n$-th iterate
$f^{\circ n}$, extends to a rational map between the corresponding
projective spaces $\oRatd$ and $\oRatdn$.
According to DeMarco \cite{DeMarco05}, the map $\Phi_n$ has a non-empty indeterminacy locus $I(d) \subset \partial \Ratd$ which is independent of $n$. 
In particular, $\Phi_n$ fails to have a continuous extension
at any point of $I(d)$. 
To resolve
this situation, it is convenient to consider the embedding 
$$
\begin{array}{cccc}
  \Phi:&\Ratd&\hookrightarrow&\Pi_{n\ge1} \oRatdn\\
       & f & \mapsto& (f, f^{\circ 2},\cdots).
\end{array}
$$
This yields a larger compactification of $\Ratd$:
$$\wRatd := \operatorname{closure} (\Phi(\Ratd)) \subset \Pi_{n\ge1} \oRatdn.$$
Trivially, every iterate map $\Phi_n$ extends continuously to $\wRatd$ as the projection to the $n$-th coordinate.

On the other hand, every rational map $f \in \Ratd$ has a unique measure
of maximal entropy $\mu_f$, see \cite{Freire83,Ljubich83}. This measure is also characterized as
the unique probability measure such that
$$\dfrac{1}{d^n} (f^{\circ n})^* \mu \to \mu_f,$$
for all 
probability measures $\mu$ (see \S \ref{s:mme}) having no mass on the exceptional set of $f$. The support of $\mu_f$ is the Julia set of $f$, denoted by $\cJ(f)$. 
Moreover, $\mu_f \in M^1(\oC)$ depends continuously on $f\in \Ratd$, where
$M^1(\oC)$ denotes the space of probability measures endowed with the weak*-topology, see \cite{Mane83}.

Given $g \in \partial \Ratd$, the pull-back operator on  (non-exceptional) measures extends continuously to $g$. More precisely, the \emph{pull-back $g^*$} can be defined by the property that,
for all non-exceptional  probability
measures $\mu$ for $g$,
as $f\in \Ratd$ converges to $g$, we have that $f^* \mu$ converges to $g^* \mu$, (see \S \ref{s:pull-back}).

Our first result, extending DeMarco's and DeMarco-Faber's results
(\cite[Theorem~0.1]{DeMarco05} and \cite[Theorem A]{DeMarco14}), is about
the limiting behavior at $\partial \Rat_d$ of the measure of maximal entropy:

\begin{introtheorem}
  \label{A}
  For any $\mathbf{g}=(g_n)_{n\ge 1} \in \wRatd$, there exists a probability measure
  $\mu_{\mathbf{g}}\in M^1(\oC)$ such that if  $\mu\in M^1(\oC)$ is any non-exceptional  for $\mathbf{g}$, then 
  as $n \to \infty$,
  $$\dfrac{1}{d^n}(g_n)^* \mu \to \mu_{\mathbf{g}}.$$
  Moreover, the map
  $$
  \begin{array}{cccc}
    \wRatd&\to& M^1(\oC)\\
      \mathbf{g} & \mapsto & \mu_{\mathbf{g}}
  \end{array}
  $$
  is continuous. 
\end{introtheorem}

Note that $\oRatd \setminus I(d)$ can be identified with
a subset of $\wRatd$ via the continuous extension
of  $\Phi: \Ratd \to \wRatd$.
After this identification, the previous theorem
for $\mathbf{g} \in \oRatd \setminus I(d)$ was established
by DeMarco in \cite[Theorem 0.1]{DeMarco05}.

The limiting behavior of measures at the indeterminacy
locus $I(d)$ was further studied by
DeMarco and Faber in \cite{DeMarco14}.
They showed that any limit of the measures $\mu_f$
as $\Ratd \ni f \to g \in \partial \Ratd$ is purely atomic, see \cite[Theorem A]{DeMarco14}.
We build on their techniques to prove Theorem \ref{A}.

Degeneration along holomorphic families can be studied
with the aid of non-Archimedean dynamics. A holomorphic family of maps $f_t \in \oRatd$, parameterized by $t$ in a
neighborhood of $0 \in \C$, is \emph{degenerate} if $f_t \in \partial\Ratd$ only for $t=0$. DeMarco and Faber showed that along any
degenerate holomorphic
family the measures of maximal entropy converge, see \cite[Theorem B]{DeMarco14}.
Moreover, they proved that the limit measure is keen to
an equidistribution measure associated to a
non-Archimedean dynamical system. Related approaches  studying
the limit of measures and the corresponding Lyapunov exponents
were implemented by Favre \cite{Favre20}, Favre-Cong \cite{Favre24}, and Okuyama \cite{Okuyama20, Okuyama24}.
It is worth to mention that, for every degenerate family $f_t$,
there exists $\mathbf{g}=(g_n) \in \wRatd$ such that
$f_t^{\circ n} \to g_n$, as $t \to 0$, for all $n$. Thus, convergence
of the measure of maximal entropy along degenerate holomorphic families can be deduced from Theorem \ref{A}. Although the proof
of Theorem~\ref{A} profits from non-Archimedean dynamics intuition,  it does not rely on  non-Archimedean tools.

There is a corresponding discussion in the \emph{moduli space}, which was also initiated by DeMarco, in \cite{DeMarco07}.
Consider the quotient $\ratd:=\Ratd/\sl2c$ of $\Ratd$
that identifies M\"obius conjugate rational maps.
Following Silverman \cite{Silverman98}, its GIT-compactification $\oratd$ is a projective variety and the boundary $\partial\ratd:=\orat_d\setminus\rat$ consists of the GIT-conjugacy classes of semistable maps in $\oRatd$. For $n\ge 1$, the $n$-th iterate map $\Psi_n: \rat_d\to \rat_{d^n}$, sending $[f]$ to $[f^{\circ n}]$, 
extends to a rational map $\oratd \dashrightarrow \oratdn$
with non trivial indeterminacy locus that depend on $n$, see \cite[Theorem 5.1]{DeMarco07} and \cite[Theorem A]{Kiwi23}.
To resolve this situation, 
as above, we consider $$
\begin{array}{cccc}
  \Psi:&\ratd&\hookrightarrow&\Pi_{n\ge1} \oratdn\\
       & [f] & \mapsto& ([f],[f^{\circ 2}],\cdots),
\end{array}
$$
and let
$$\wratd := \operatorname{closure} (\Psi(\ratd)) \subset \Pi_{n\ge1} \oratdn.$$
Iterate maps $\Psi_n$ extend continuously to $\wratd$ as the projection to the corresponding coordinate.

In order to consider measures for elements in the moduli space, following DeMarco \cite{DeMarco07}, it is convenient to work with
 \emph{conformally barycentered probability measures}, see \S \ref{s:moduli} .
 The space of conformally barycentered  probability measures modulo  the push-forward action by $\operatorname{SO}(3)$ is locally compact. We denote by $\overline{M}^1_{dm}(\oC)$ its one point ($=[\infty]$) compactification.

 For any measure $\mu \in M^1(\oC)$ with no atoms of weight $\ge 1/2$, there exists a M\"obius transformation such that $\nu = A_\ast \mu$
 is conformally barycentered. Moreover, $\nu$ is unique up to push-forward by an element in $\operatorname{SO}(3)$. Thus, we have a well defined
 class $[\mu] :=[\nu] \in  \overline{M}^1_{dm}(\oC)$.
 If  $\mu$ has an atom of weight $\ge 1/2$, we let
 $[\mu] := [\infty]\in  \overline{M}^1_{dm}(\oC)$. 
 Since the measure of maximal entropy $\mu_f$ of any rational map $f$
 is atom free, there exists
 a continuous map 
 $$
\begin{array}{ccc}
  \ratd&\rightarrow&\overline{M}^1_{dm}(\oC)\\
     f & \mapsto& [\mu_{f}]. 
\end{array}
$$
We show this map extends continuously to $\wratd$, answering in the positive the question posed 
by DeMarco in \cite[Section 10]{DeMarco07}:  

\begin{introtheorem}
  \label{B}
   For any $[\mathbf{g}]=([g_n])_{n\ge 1} \in \wratd$, there exists 
  $x_{[\mathbf{g}]} \in \overline{M}^1_{dm}(\oC)$ such that if    $\mu\in M^1(\oC)$ is non-exceptional for $(g_n)_{n\ge 1}$, then as $n \to \infty$, 
  $$\left[\dfrac{1}{d^n}(g_n)^* \mu\right] \to x_{[\mathbf{g}]}. $$
  Moreover, the map
  $$
  \begin{array}{cccc}
    \wratd&\to& \overline{M}^1_{dm}(\oC)\\
      {[\mathbf{g}]} & \mapsto & x_{[\mathbf{g}]} 
  \end{array}
  $$
  is continuous. 
\end{introtheorem}

The space  $\wratd$ coincides with the compactification of $\ratd$ constructed by taking the inverse limit of the closure of $\ratd$ in $\Pi_{i=1}^n\overline{\mathrm{rat}}_{d^i}$ under the iterate map $\ratd\to\Pi_{i=1}^n\overline{\mathrm{rat}}_{d^i}$. On the other hand, by considering the closure of the graph of the map $\ratd\rightarrow\overline{M}^1_{dm}(\oC)$ in $\oratd\times\overline{M}^1_{dm}(\oC)$, DeMarco constructed a compactification $\overline{\orat}_d$ of $\ratd$ and established that, in the quadratic case, the compactifications $\widehat{\rat}_2$ and  $\overline{\orat}_2$ are canonically homeomorphic \cite[Theorem 1.1]{DeMarco07}. However, for higher degrees,  the compactifications $\wratd$ and  $\overline{\orat}_d$ are not homeomorphic and, in particular, there cannot exist a continuous map $\overline{\orat}_d\to\wratd$ which restricts to the identity on $\ratd$, see \cite[Section 10]{DeMarco07}. Our Theorem~\ref{B} immediately implies that there exists a continuous map $\wratd\to\overline{\orat}_d$ which restricts to the identity on $\ratd$. In other words,  the GIT  limits of the iterates of an unbounded sequence in $\ratd$ determine the limit
of the barycentered measures of maximal entropy.

Along a degenerating sequence of rational maps $\{f_k\}$, the choice of coordinates
in the domain and in the range become of central importance.
The key to Theorems~\ref{A} and \ref{B} is to analyze the case
in which an iterate of $f_k$
uniformly converges to a  non-degenerate rational map,
modulo changes of coordinates in the domain and the range.
We obtain a ``Structure Theorem'' that describes  the  limiting  dynamics, in
this case,  see  \S\ref{s:fullyramified}. The Structure Theorem yields, as a byproduct of independent interest, Theorem \ref{C}, stated below.

Recall that a \emph{polynomial-like map of degree $d$} is a triple $(U,V,\psi)$, where $U$ and $V$ are open topological discs contained in $\oC$ with $U$ relatively compact in $V$ and $\psi:U\to V$ is a proper holomorphic map of degree $d$. Any polynomial-like map of degree $d$ is quasi-conformally equivalent to a degree  $d$ polynomial, see \cite[Theorem 1]{Douady85}.

\begin{introtheorem}\label{C}
  Let $\{f_k\}$ be a sequence in $\Ratd$ 
  such that $\{[f_k]\}$ converges to an element in $\partial \rat_d$. 
  Assume there exist $\{A_k\}$ and $\{B_k\}$ such that
  $B_{k}\circ f_k^{\circ 5}\circ A_{k}^{-1}$ converges in $\Rat_{d^5}$, as $k\to\infty$.
  Then for any sufficiently large $k$, there exist $U_k$ and $V_k$  such that $(U_k,V_k, f_k)$ is a polynomial-like map of degree $d$. Moreover,
  $\cJ(f_k)\subset U_k$ and $\oC\setminus U_k$ is contained in a completely invariant attracting Fatou component of $f_k$.
\end{introtheorem}

In our context, the assumption that $f_k^{\circ 5}$ converges to a non-degenerate map, modulo changes of coordinates in the domain and range,
is sufficiently weak. However, it might not be sharp. Is the conclusion of
Theorem~\ref{C} true, if $5$ is replaced by $3$?

{\subsection*{Outline}
  The paper is organized as follows. In \S \ref{s:degenerate}, we discuss the basic properties of degenerate rational maps  including their action on measures via  pull-back. In \S\ref{s:fullyramified}, based on the trees of bouquets
  of spheres from Appendix \ref{s:appendix}, we establish our Structure Theorem for the limiting dynamics of ``fully ramified sequences''.  As a byproduct,
  we obtain Theorem \ref{C}.
  We prove Theorem \ref{A} in \S \ref{s:Aproof} and Theorem \ref{B} in \S \ref{s:moduli}. In Appendix~\ref{s:appendix}, we discuss maps acting on trees of bouquet of spheres. This is closely related to the work by Arfeux~\cite{ArfeuxCompactification} regarding  maps acting on trees of spheres.
  
{\subsection*{Acknowledgments}
Part of this work was done when the authors were in the Institute for Mathematical Sciences (IMS) at Stony Brook University. The authors are grateful to the IMS for its hospitality.

\section{Degenerate rational maps}
\label{s:degenerate}

A degenerate rational map $f \in \partial \Ratd$ has a well defined action 
outside a finite set of points, called the \emph{holes} of $f$.
 The action,
known as the \emph{reduction $\tilde{f}$} of $f$,
is a lower degree rational map, maybe constant.
Intuitively, the holes ``map'', under $f$, onto $\oC$ with a degree, called
its \emph{depth}. The reduction $\tilde{f}$, the set of holes
$\hole(f)$ and their depth $d_h(f)$,
completely prescribe the degenerate map $f$. We discuss these definitions from
 \cite{DeMarco05} in \S \ref{s:holes}. The parameter space of rational maps, degenerate or not,
is discussed in \S \ref{s:spaces}. A well known formula for the depth of the holes of compositions is written in \S \ref{s:composition}. As rational maps converge to a degenerate map, their action 
is better understood by changing coordinates, or scales, at the domain and
the range.  In \S \ref{s:post}, we consider the space of rational
maps, possibly degenerate but with non-constant reduction, modulo post-composition by M\"obius maps. It is a compact space.
 In \S \ref{s:depth}, we present a formula
for the depth of holes in the setting of \S \ref{s:post} as an analog of 
a formula obtained by the authors in \cite{Kiwi23}, and also furnish a formula to compute the depth using preimages.  
In \S \ref{s:pull-back}, we recall the definitions of the pull-back of measures by a rational map and, following DeMarco \cite{DeMarco05},
by a degenerate rational map. We show that the degenerate pull-back
is the continuous extension 
of the pull-back (Proposition \ref{p:pullback-continuity}). 
In \S \ref{s:mme}, we summarize some general properties of the measures
of maximal entropy.
Continuity of the pull-back of measures yields a key result established by DeMarco and Faber (see \cite[Theorem 2.4]{DeMarco14}}), which provides us with an equation satisfied by suitable limits
of the measures of maximal entropy.

\subsection{Holes and reduction}
\label{s:holes}
Although we mostly work in non-homogeneous coordinates, to define reductions, holes and depths, homogeneous coordinates are convenient.
A \emph{rational map $f: \P^1  \to \P^1 $ of degree $d\ge 2$} can be written as
$$f: [z:w] \mapsto [P(z,w):Q(z,w)],$$
where $P,Q$ are relatively prime
degree $d$ homogeneous polynomials.
For concreteness, one may write
\begin{equation}
  \label{eq:pq}
    \begin{aligned}
  P(z,w) & = & a_0 w^d + a_1 z w^{d-1} + \cdots +  a_d z^d,\\
  Q(z,w) & = & b_0 w^d + b_1 z w^{d-1} + \cdots + b_d z^d.
  \end{aligned}
\end{equation}
Then $f = [P:Q]$ is a degree $d$ rational map if and only if
$P,Q$ are relatively prime.

A degree $d$ \emph{degenerate rational map}
$f$ is an expression $[P:Q]$ such that
$P,Q$, as above,  are not relatively prime. Consistent with the notation,
for all $\lambda \in \C \setminus \{0\}$, we regard $[P:Q]$
and $[\lambda P: \lambda Q]$ as the same map. 
Let $H_f$ be a common factor of $P$ and $Q$
of maximal degree. Then, there exist polynomials $\tilde{P}$, $\tilde{Q}$,
such that 
$$P = H_f \tilde{P}\ \ \ \  \text{and}\ \ \ \ Q= H_f \tilde{Q}. $$
We say that $$\hole (f) := \{ h \in \P^1  : H_f(h) =0\}$$
is the \emph{set of holes of $f$}. The \emph{depth $d_h(f)$} of a hole $h$ is
the multiplicity of $h$ as a root of $H_f$; that is, modulo multiplication by a constant,
$$H_f = \prod_{h \in \hole(f)} L_h^{d_h(f)},$$
where $L_h$ is a homogeneous linear form vanishing at $h$; moreover, for notational convenience, we let
$d_h(f):=0$ whenever $h \in\P^1\setminus \hole(f)$.
We say that $$\tilde{f} := [\tilde{P}: \tilde{Q}]$$
is the \emph{reduction of $f$}. 
Note that $\tilde{f}$ is a well defined self-map of $\P^1 $.
We write $f=H_f \cdot \tilde{f}$.

Given a non-degenerate rational map $\tilde{f}$  of degree $\deg \tilde{f} < d$ and a set of holes $\hole(f)$
with prescribed depths such that the total number of holes, counted with depths,
is $d-\deg \tilde{f}$, there exists a
unique degree $d$ degenerate map $f$ compatible with
this information.

For notational simplicity, we mostly work with non-homogeneous coordinates,
after identification of $\P^1 $ with $\oC := \C\cup\{\infty\}$.
Thus, we write $f(z) = P(z,1)/Q(z,1)$ or simply $f(z) = P(z)/Q(z)$. 

\subsection{Spaces of rational maps}
\label{s:spaces}
The space of possibly degenerate rational maps $\oRatd$ is the $(2d+1)$-dimensional complex projective space $\P^{2d+1} $. Its elements are identified with
expressions of the form $f = [P:Q]$ by 
$$
\begin{array}{ccc}
  \oRatd & \leftrightarrow & \P^{2d+1}  \\
   f = [P:Q] & \leftrightarrow & [a_0: \cdots : a_d : b_0 : \dots :b_d],
\end{array}
$$
where $P,Q$ are polynomials of the form \eqref{eq:pq}.

The degree $d$ homogeneous resultant of $P,Q$ is a homogeneous polynomial
in the coefficients. Its vanishing locus $\partial\Ratd := \{ \res = 0\}$
is the set of {degenerate maps} in $\oRatd$.
Hence, the Zariski open subset $\Rat_d := \oRatd \setminus \partial \Ratd$
is the space of rational maps of degree $d$. Convergence in $\Ratd \subset \P^{2d+1}$ coincides with uniform convergence of rational maps.

M\"obius transformations form the space $\Rat_1$ of degree $1$ rational maps that is naturally identified with $\operatorname{PSL(2,\C)}$. Each element of $\partial \Rat_1$ is a degenerate M\"obius transformation: a map with constant reduction and one hole, counted with depth.

\subsection{Composition of degenerate maps}
\label{s:composition}
Given $d, d' \ge 1$, consider the map 
$$
\begin{array}{cccc}
  \circ :& \Ratd \times \Rat_{d'} & \to & \Rat_{dd'}\\
  & ( f, g) & \mapsto & f \circ g.
\end{array}
$$
It extends to a rational map from the projective
variety $\oRatd \times \oRat_{d'}$ into $\oRat_{dd'}$ with indeterminacy locus
$$I(d,d') := \{ (f,g) : \tilde{g} \equiv h \in \hole(f) \}.$$

Following \cite[Lemma 2.6]{DeMarco07}, for all $(f,g) \in \oRatd \times \oRat_{d'} \setminus I(d,d')$,
we have that $f \circ g$ is the unique element of $\oRat_{dd'}$ with
reduction $\tilde{f} \circ \tilde{g}$ and 
holes $h\in\tilde{g}^{-1} (\hole(f)) \cup \hole(g)$ with depths 
\begin{equation}
  \label{eq:depth-composition}
  d_h (f\circ g) = d_h (g) \cdot d + \deg_h (\tilde{g}) \cdot d_{\tilde{g}(h)} (f).
\end{equation}


\subsection{Left classes of rational maps}
\label{s:post}
To deal with degenerations towards maps with constant reduction,
the key is the following
``scaling'' result:

\begin{proposition}[{\cite[Lemma 2.1]{DeMarco14}}]
  \label{p:post-quotient}
  Let $\{f_k\}$ be a sequence in $\Ratd$. Then there exists a subsequence
  $\{f_{k_i} \}$ and a sequence $\{A_i\}$ in $\operatorname{PSL(2,\C)}$ 
  such that
  $$A_i \circ f_{k_i} \to \varphi \in \oRatd,$$
  and  $\tilde{\varphi}$ is not constant.
  Moreover, if $\{B_i\}$ is another such sequence, then $\{B_i^{-1} \circ A_i\}$
  converges in $\operatorname{PSL(2,\C)}$.
\end{proposition}

This result may be understood as the compactness of certain
orbit space. 
Let us denote the maps with non-constant reduction by 
$$\Rat_d^*:=\{ f \in \oRatd : \deg \tilde{f} \neq 0 \}.$$
Then post-composition by $\operatorname{PSL(2,\C)}$ leaves $\Rat_d^*$ invariant.
An equivalent formulation of the above proposition is the following:

\begin{corollary}\label{cor:cmpt}
  Each $\operatorname{PSL(2,\C)}$-orbit, under post-composition, 
  is closed in $\Rat_d^*$.
  The quotient topological space $\llbracket \Rat_d^*]:=\operatorname{PSL(2,\C)} \backslash \Rat_d^*$
is compact. 
\end{corollary}

\begin{remark}
  In fact, $\Rat_d^*$ is the set of  GIT-stable elements of $\oRat_d$, under post-composition by  $\operatorname{PSL(2,\C)}$. There are no strictly semi-stable
  elements for this action. 
  Hence, $\llbracket \Rat_d^*]$ is a geometric quotient which
  is, naturally, a projective variety.
\end{remark}

We  denote the
$\operatorname{PSL(2,\C)}$-orbit of $f \in \Rat_d^*$ under the
post-composition action by 
$$\llbracket f ] \in \llbracket \Rat_d^*].$$

Let $\{f_k\}\subset\Rat_d$ be a sequence such that $\llbracket f_k^{\circ n}]$ converges to $\llbracket\varphi_n]$ in $\llbracket \Rat_{d^n}^*]$ for each $n\ge 1$. 
We say that 
 $\{ A_{n,k}: n,k \ge 1 \} \subset \Rat_1$ 
 is a \emph{collection of post-scalings for $\{f_k\}$} if for all $n \ge 1$,
 as $k \to \infty$, 
   $$A_{n,k} \circ f^{\circ n}_k \to \varphi_n.$$
 The following result is similar to \cite[Lemma 2.5]{DeMarco14}. 
\begin{lemma}
  \label{l:decomposition}
  Assume that $\llbracket f_k^{\circ n}]$ converges to $\llbracket\varphi_n]$ in $\llbracket \Rat_{d^n}^*]$ for each $n\ge 1$. 
  Then for all $n\ge 1$, there exist $\varphi_{n,n+1} \in \Rat_d^*$ such that
  $$\varphi_{n+1} = \varphi_{n,n+1} \circ \varphi_n. $$ 
  Moreover,  if $\{A_{n,k}\}$ is a collection of post-scalings for $\{f_k\}$,
  then
  $$A_{n+1,k} \circ f_k \circ A^{-1}_{n,k} \to \varphi_{n,n+1}.$$
\end{lemma}

\begin{proof}
  For any $n \ge 1$,  
  let $\{ A_{n,k}\} \subset \Rat_1$
  be a collection of scalings for $\{f_k\}$. 
  Modulo passing to a subsequence,
  by Proposition \ref{p:post-quotient} applied
  to $\{ f_k \circ A^{-1}_{n,k}\}$, there exists $\{B_{n+1,k}\}$
  such that, as $k \to \infty$,
  $$B_{n+1,k} \circ f_k \circ A^{-1}_{n,k} \to \psi_{n,n+1},$$
  for some $\psi_{n,n+1} \in \Rat^*_{d}$.
  Continuity of composition (\S \ref{s:composition}) implies that
  $B_{n+1,k} \circ f_k^{\circ n+1} \to \psi_{n,n+1} \circ \varphi_n$.
  Again by Proposition \ref{p:post-quotient}, we conclude that 
  $A_{n+1,k} \circ B^{-1}_{n+1,k} \to M_{n+1}$,
  for some M\"obius transformation $M_{n+1} \in \Rat_1$.
  Let $\varphi_{n,n+1} := M_{n+1} \circ \psi_{n,n+1}$.
  Thus, a map $\varphi_{n,n+1} \in \oRatd$ as in the statement of the Lemma
  exists.

  Note that any accumulation point $\varphi \in \oRatd$
  of $A_{n+1,k} \circ f_k \circ A^{-1}_{n,k}$ satisfies that $\varphi \circ \varphi_n = \varphi_{n+1}$.
  To finish the proof, it suffices to show that $\varphi$ is uniquely determined by this property, i.e. $\varphi = \varphi_{n,n+1}$.
  Indeed, surjectivity   of $\tilde{\varphi}_n$ yields that $\tilde{\varphi} = \tilde{\varphi}_{n,n+1}$. Moreover, the depth $d_h(\varphi)$, for any $h \in \oC$,
  is completely determined by  the equation 
  $$d_h(\varphi_{n+1}) = d_h(\varphi_n) \cdot d + \deg_h \tilde{\varphi}_n \cdot d_{\tilde{\varphi}_n(h)}(\varphi).$$
  Therefore,
  $\varphi = \varphi_{n,n+1}$.
  \end{proof}

\subsection{Depth formulas}
\label{s:depth}

The following result is the complex analytic analogue of \cite[Lemma 2.9]{Kiwi23}. 

\begin{lemma}\label{lem:s}
  Let $\{f_k\} \subset \Ratd$ be a sequence  converging to $g\in\partial{\mathrm{Rat}}_d$  with constant reduction.
  Assume that
  $\llbracket f_k] \to \llbracket \varphi]$. Then there exists
  $a \in \oC$ such that
  $$\hole(g) = \hole(\varphi) \cup \tilde{\varphi}^{-1} (a).$$
  Moreover, $$d_h(g) =
  \begin{cases}
    d_h(\varphi), & \text{ if } h \notin \tilde{\varphi}^{-1}(a), \\
    d_h(\varphi) + \deg_h \tilde{\varphi}, & \text{ otherwise.}
  \end{cases}
  $$
  Furthermore, if $A_k \circ f_k \to \varphi$ in $\Rat_d^*$, then $A_k  \to A\in\partial\Rat_1$ with $\tilde{A}=a$ and $\hole(A)=\{c\}$ where
  $\tilde{g} \equiv c$.
\end{lemma}
\begin{proof}
  Consider a sequence $\{A_k\} \subset \Rat_1$ such that $A_k \circ f_k \to \varphi$.
  Passing to a subsequence, suppose $A_{k_i} \to A \in \partial \Rat_1$ and write $a:=\tilde{A}$.
  Then $\hole(A) = \{c\}$ where $\tilde{g} \equiv c$; for otherwise $A_{k_i} \circ f_{k_i}(z)$ would converge
  to $a$ outside $\hole(\varphi)$, and hence $\tilde{\varphi}$ would be constant, contrary to our hypothesis.

  It follows that $A^{-1}_{k_i}$ converges to some $B\in\partial\Rat_1$ with reduction $c$ and hole $a$, i.e. $H_B(z) = z-a$. Hence,
  $$A^{-1}_{k_i} \circ (A_{k_i} \circ f_{k_i}) \to (H_B \circ \tilde{\varphi}) H_\varphi  \tilde{\varphi} \cdot c = g.$$
  Thus, the first two assertions hold. Since $a$ is independent of the subsequence,  the last assertion also holds. 
\end{proof}

Given $g \in \partial \Rat_d$, we define the  \emph{algebraically exceptional set}  $\ex_g\subset\oC$  by 
$$
\ex_g:=
\begin{cases}
\{\tilde g\}\ \ &\text{if}\ \deg\tilde g=0,\\
\emptyset\ \ &\text{if}\ \deg\tilde g\ge 1; 
\end{cases}$$ 
The next lemma provides a way to compute depths by counting preimages
of non-exceptional points. It states precisely the sense in which 
holes ``map  onto $\oC$'' by a degenerate map (c.f.  \cite[Lemmas 4.5 and 4.6]{DeMarco05}). 

\begin{lemma}\label{lem:pre-number}
Let $\{f_k\} \subset \Ratd$ be a sequence  converging to $g\in\partial{\mathrm{Rat}}_d$ and 
  assume that
  $\llbracket f_k] \to \llbracket \varphi]$.
  Pick $z_0\in \oC$. Then 
  for any $w \in \oC$, there exists an arbitrarily small neighborhood $U$
  of $w$, such that the following hold for all  sufficiently large $k$:
  \begin{enumerate}
  \item  If $z_0\not\in \ex_g$, then, counting with multiplicities,
  $$\#(f_k^{-1} (z_0) \cap U) =
  \begin{cases}
  d_w(g) &
\text{ if } \tilde{g}(w) \neq z_0, \\
  d_w(g) + \deg_w \tilde{g} & \text{ if } \tilde{g}(w) = z_0.
  \end{cases}
 $$  
 \item Counting with multiplicities,
  $$\#(f_k^{-1} (z_0) \cap U)\le  d_w(\varphi)+\deg_w\tilde\varphi\le d_w(g)+\deg_w\tilde\varphi.$$
  
 \end{enumerate} 
\end{lemma}

\begin{proof}
  Let $P,Q \in \C[z]$ be two degree $d$ polynomials such that $g=P/Q$.
  Let $H$ be a  common divisor of $P$ and $Q$ with maximal degree.
  Then
  $P=H\tilde{P}$, $Q= H \tilde{Q}$ and  $\tilde{g} =\tilde{P}/\tilde{Q}$ for some relatively
  prime polynomials $\tilde{P},\tilde{Q}$. If $z_0 \notin \ex_g$, we have that $\tilde{g} \not\equiv z_0$. Hence, the multiplicity of $w$ as a solution of 
  $$0=P-z_0Q=H(\tilde{P}-z_0 \tilde{Q})$$
  is $d_w(g)$ if $\tilde{g}(w) \neq z_0$, and  $d_w(g) + \deg_w \tilde{g}$ if  $\tilde{g}(w)=z_0$. 
  Thus writing $f_k=P_k/Q_k$ and observing that the number of solutions
  of $P_k - z_0Q_k$ in a small neighborhood of $w$ converge
  to the multiplicity of $w$ as a solution of $P-z_0Q=0$, we obtain statement (1). 
  
  In view of the above, for statement (2), it suffices to consider the case where $\deg\tilde g=0$.  Let $\{A_k\}\subset\Rat_1$ such that $A_k\circ f_k\to\varphi$. Passing to a subsequence in $k$, we can assume that $A_k(z_0)\to a$. Then given a small neighborhood $U$ of $w$, we have that, for all  sufficiently large $k$,
  $$\#(f_k^{-1} (z_0) \cap U)=\#((A_k\circ f_k)^{-1}(A_k(z_0))\cap U)=
  \begin{cases}
  d_w(\varphi)+\deg_w\tilde\varphi\ &\text{if}\ a=\tilde\varphi(w),\\
  d_w(\varphi) \ &\text{if}\  a\not=\tilde\varphi(w),\\
  \end{cases}$$
  counted with multiplicities. 
  Then statement (2) follows from Lemma \ref{lem:s}.
\end{proof}


\subsection{Pull-back of measures}
\label{s:pull-back}
Given a continuous function $\varphi : \oC \to \R$ and
a rational map $f:\oC\to\oC$, the \emph{push-forward} of $\varphi$ by $f$ is
the continuous function:
$$f_* \varphi (z) := 
\begin{cases}
\sum\limits_{w \in f^{-1}(z)} \varphi(w)\ \ &\text{if}\ \deg f \ge 1,\\
\ \ \ 0 \ &\text{if}\ \deg f=0.
\end{cases}$$
Clearly, $f_*$ is a non-negative bounded linear operator on the space of
continuous function with the $\sup$-norm.
Its dual is the \emph{pull-back $f^*$}, whose action
on a (positive Radon) measure $\mu$ in $\oC$
is determined by the property that
$$\int \varphi(w) d f^* \mu (w) = \int f_* \varphi (z) d \mu(z),$$
for all continuous functions $\varphi$.
Taking $\varphi \equiv 1$, yields $f^* \mu (\oC) = \deg f \cdot \mu(\oC)$.
All measures considered here are non-negative Radon measures.

For any $g \in \partial \Rat_d$,  we can obtain a natural measure of total mass $d-\deg{\tilde{g}}$ induced by  the depths of its holes:
\begin{definition}\label{def:depthmeasure}
Given $g \in \partial \Rat_d$, we say that the measure 
$$\eta_g := \sum_{h \in \hole(g)} d_h(g) \delta_h$$
  is the \emph{depth measure} of $g$.
\end{definition}

Consider $g \in \oRatd$. 
If $g\in\partial\Ratd$, recall from \S\ref{s:depth} the algebraically exceptional set  $\ex_g\subset\oC$ for $g$. If $g \in \Ratd$, then the dynamically exceptional set $E_g$ of $g$, is the one formed by all finite grand orbits, under $g$. 
We say  that a measure $\mu$ in $\oC$ is \emph{non-exceptional} for $g \in \oRatd$ if $\mu(\ex_g) =0$ when $g \in \partial \Ratd$ and $\mu(E_g)=0$ when
$g \in \Ratd$. 
Moreover, given an element in $\Pi_{i=1}^\infty\oRat_{s_i}$, we say a measure $\mu$ in $\oC$ is non-exceptional for this element if $\mu$ is non-exceptional for each coordinate. Following DeMarco \cite[Section 3]{DeMarco05}, we introduce the pull-back of measures by $g$.  
\begin{definition}
  Given $g \in \partial \Rat_d$ and a non-exceptional measure $\mu$ for $g$, 
  we say that
  $$g^* \mu = \tilde{g}^* \mu + \mu(\oC) \cdot \eta_g $$
  is the \emph{pull-back of $\mu$} under $g$. 
  \end{definition}
Given $g \in \partial \Rat_d$ and a continuous function $\varphi : \oC \to \R$, we define the \emph{push-forward of $\varphi$ by $g$} 
$$g_* \varphi :=
\tilde{g}_*  \varphi  + \sum_{h \in \hole(g)} d_h(g) \varphi(h);$$
then  for any a non-exceptional measure $\mu$ for $g$, one obtains that 
$$\int \varphi d g^* \mu = \int g_* \varphi d \mu.$$

Although not explicitly discussed in DeMarco \cite{DeMarco05} and DeMarco-Faber \cite{DeMarco14},
a key fact is that
the pull-back defined above is the continuous extension
of the pull-back of (non-exceptional) measures:

\begin{proposition}
  \label{p:pullback-continuity}
  Let  $g \in \partial\Ratd$ and $\mu$ be a
  non-exceptional measure for $g$.
  Then  $f^*\nu$ converges to $g^* \mu$, as $f \in \Rat_d$ converges to $g$ and
  $\nu$ converges to $\mu$.
\end{proposition}

\begin{proof}
Let $W\subset\oC$ be a small neighborhood of the point in $\ex_g$ if  $\ex_g\not=\emptyset$; and set $W=\emptyset$ if $\ex_g=\emptyset$. Consider
$z \in \oC \setminus W$.
As $\Ratd \ni f \to g$, the set $f^{-1}(z)$ converges
to $\hole(g) \cup \tilde{g}^{-1}(z)$; each point is counted with multiplicity. 
 More precisely, given $w \in \hole(g) \cup \tilde{g}^{-1}(z)$, let $m_w := d_w(f)$ if
$w \notin \tilde{g}^{-1}(z)$, and $m_w := d_w(f) + \deg_w \tilde{g}$ otherwise. By Lemma \ref{lem:pre-number} (1), given any small neighborhood of $w$,   for all $f$ sufficiently
close to $g$, this neighborhood contains $m_w$
preimages of $z$ under $f$, counted with multiplicities.

Now, consider  a nonnegative continuous test function $\varphi:\oC\to\R$.
From the previous paragraph, 
as $f \to g$, we have that $f_* \varphi: \oC \setminus W \to \R $ pointwise converges to
the continuous function $g_*\varphi :\oC \setminus W \to \R$.  
Since $\oC \setminus W$ is compact, the convergence is, in fact, uniform.

 Given $\varepsilon >0$,
assume that $W$ is sufficiently small with $\mu(\partial W)=0$ and $\nu$ is sufficiently close
to $\mu$ so that both $\mu(W)$ and $\nu(W)$ are less than $\varepsilon$. 
Then 
  \begin{eqnarray*}
    \left|\int_{\oC} \varphi  (f^* d\nu -   g^* d\mu) \right| & =
    &\left|\int_{\oC} f_* \varphi d\nu - \int_{\oC} g_* \varphi d\mu \right|\\
    & =& \left|\int_{\oC \setminus W} f_* \varphi d\nu - \int_{\oC \setminus W} g_* \varphi d\mu + \int_{W}  f_* \varphi d\nu -  \int _{W}g_* \varphi d\mu \right|\\
    &\le& \left|\int_{\oC \setminus W} f_* \varphi d\nu - \int_{\oC \setminus W} g_* \varphi d\mu\right| + \left|\int_{W}  f_* \varphi d\nu\right|+\left|  \int _{W}g_* \varphi d\mu \right|\\
    &\le&\left|\int_{\oC \setminus W} f_* \varphi d\nu - \int_{\oC \setminus W} f_* \varphi d\mu\right|+\left| \int_{\oC \setminus W} (f_* \varphi-g_* \varphi) d\mu\right| + d \|\varphi\|_\infty 2 \varepsilon, 
  \end{eqnarray*}
  where the last inequality follows from the fact that $\|F_* \varphi\|_\infty \le d \|\varphi\|_\infty$ for any $F \in \oRatd$.
  The convergence of $\nu$ to $\mu$ and the uniform convergence of $f_* \varphi $ to $g_* \varphi $ in $\oC \setminus W$, as $f \to g$, yield that $f^* \nu \to g^* \mu$
  as $(f,\nu) \to (g,\mu)$.  
\end{proof}

\subsection{Measures of maximal entropy}
\label{s:mme}
Recall that $M^1(\oC)$ denotes the space of probability measures in $\oC$ endowed with the weak* topology. 
Given a rational map $f \in \Ratd$, there exists a unique probability
measure $\mu_f$ such that $f^* \mu_f = d \mu_f$ and $\mu_f(\cE_f)=0$ where $\cE_f$ is
the exceptional set for $f$
(i.e. the elements of $\oC$ with finite grand orbit).
Moreover, for any $\mu\in M^1(\oC)$ satisfying $\mu (\cE_f)=0$, 
$$\dfrac{1}{d^n} (f^{\circ n})^\ast \mu \to \mu_f.$$ Furthermore, $\mu_f$ is supported on the Julia set $\cJ(f)$ of $f$ and is the unique measure of maximal
entropy for $f$. This measure $\mu_f$ varies continuously with $f$. See~\cite{Freire83, Ljubich83, Mane83}. 

The fundamental property of an accumulation point $\mu$
of measures $\mu_f$ as $f$ approaches
a degenerate map was formulated by DeMarco and Faber \cite[Section 2.4]{DeMarco14} in terms
of pairs of probability measures $(\mu,\nu)$ that take into account the
changes of scales involved. More precisely,
let $(\mu ,\nu)$ be a pair of probability measures on $\oC$. For a sequence $\{A_k\}\subset\Rat_1$,   we say that a  sequence $\{\mu_k\}\subset M^1(\oC)$ \emph{converges $\{A_k\}$-weakly to $(\mu, \nu)$} if $\mu_k$ converges to $\mu$ and $(A_k)_\ast\mu_k$ converges to $\nu$.

\begin{proposition}[{\cite[Theorem 2.4]{DeMarco14}}]\label{prop:mea-CE}
  Let $\{f_k\}$ be a sequence in $\mathrm{Rat}_d$ and let $\{A_k\}$  be a sequence in $\Rat_1$ such that $A_k\circ f_k$ converges to map $\varphi \in \Rat_d^*$. 
 Assume  that the  $\{A_k\}$-weak limit of $\{\mu_{f_k}\}$ is
  $(\mu,\nu)$. Then
$$\mu =\frac{1}{d}\varphi^\ast\nu.$$
\end{proposition}

\begin{proof}
  For all $k \ge 1$, 
  $$(A_k \circ f_k)^* (A_k) _* \mu_k = d \cdot \mu_{f_k}.$$
Continuity of the pull-back at
  $(\varphi,\nu)$ (Proposition \ref{p:pullback-continuity}) yields  $\varphi^* \nu = d \cdot \mu .$
\end{proof}

\section{Limiting Dynamics of Fully Ramified Sequences}
\label{s:fullyramified}
Recall from \S \ref{s:post} that $\Rat^*_d$ denotes the space of degree
$d$ rational maps with non-constant reduction. 
M\"obius transformations act by post-composition on $\Rat^*_d$.
Its orbit space, denoted by $\llbracket \Rat_d^* ]$, is a compact Hausdorff space, see Corollary \ref{cor:cmpt}.

Here we consider sequences
$\{f_k\} \subset \Ratd$ 
such that $\{\llbracket f^{\circ n}_k]\}$
converges to $\llbracket \varphi_n] \in \llbracket \Rat_{d^n}^* ]$, for all $n$.
From Lemma~\ref{l:decomposition},  there exists
$\varphi_{n,n+1} \in \Rat_d^*$ such that $\varphi_{n+1} = \varphi_{n,n+1} \circ \varphi_n$, for all $n$. We say that $n_0$ is a \emph{fully ramified time} if
$\varphi_{n_0,n_0+1}$ is a non-degenerate map of degree $d$.



In this section, we analyze the limiting dynamics for sequences having several successive fully ramified times. Our analysis is based on the trees of bouquet of spheres discussed in Appendix~\ref{s:appendix} and closely related to the work by Arfeux~\cite{ArfeuxCompactification}.
We establish the Structure Theorem \ref{p:structure} for sequences $\{f_k\}$ diverging in moduli space $\ratd$ and possessing at least $5$ successive fully ramified times. We also deduce Corollaries \ref{coro:main}  and \ref{coro:main1} which are crucial in our proofs of Theorems \ref{A} and \ref{B}. Moreover, we obtain Theorem \ref{C} as a byproduct of the Structure Theorem \ref{p:structure}.

\subsection{Statement of the Structure Theorem}

Let $\{f_k\} \subset \Ratd$ be such that  $\llbracket f_k^{\circ n} ]$ converges to $\llbracket \varphi_n ]$ in $\llbracket \Rat_{d^n}^*]$ for each $n\ge 1$.
Recall from \S \ref{s:post} that
$\{ A_{n,k}: n, k \ge 1 \} \subset \Rat_1$ 
is a collection of post-scalings for $\{f_k\}$} if for any $n \ge 1$, as $k \to \infty$, 
$$A_{n,k} \circ f^{\circ n}_k \to \varphi_n.$$
 Such a collection  $\{A_{n,k}\}$ has \emph{convergent changes of coordinates} if  for all $n \neq n'$, 
  $$A_{n',k} \circ A^{-1}_{n,k} \to A_{(n,n')} \in \oRat_1.$$
  When clear from context, $A_{(n,n')} \in \oRat_1$ will be the limit of the corresponding coordinate changes $A_{n',k} \circ A^{-1}_{n,k}$ and, if $A_{(n,n')} \in \partial\Rat_1$, write $$a_{n,n'}:=  \tilde{A}_{(n,n')}\in\oC.$$
 In the case that $A_{(n,n')} \in \partial \Rat_1$ for all $n \neq n'$, we say
 that $\{A_{n,k}: n,k \ge 1\}$ is a collection of \emph{independent post-scalings}. Note that $a_{n',n}$ is the unique hole of $A_{(n,n')}$.
By definition, independent collections of scalings have convergent changes of coordinates.


  Let $m \ge 1$ and consider a collection $\{A_{n,k}: n \ge m \}$ 
  of independent scalings, where we abuse of notation and omit
  writing $k \ge 1$ and, sometimes, simply write $\{A_{n,k}\}$. 
  In view of Appendix \ref{s:appendix}, given $\ell \ge m$, 
the collection $\{A_{n,k}: m \le n \le \ell\}$  is, in a certain sense, a tree of bouquets of spheres. More precisely,
  for each integer $n \ge 1$,
  consider a copy $\oC_n$ of the Riemann sphere.
  Let $\sim$ be the equivalence relation in $\oC_m \sqcup \cdots \sqcup \oC_\ell$ that identifies $a_{n,n'} \in \oC_{n'}$ with $a_{n',n} \in \oC_n$   if and only
  if $$a_{n,n''} = a_{n',n''}$$
  for all $n'' \neq n, n'$ such that $m \le n''\le \ell$.
  All other $\sim$-classes are trivial.
  Let
  $$\cS_{\ell}:= \oC_m \sqcup \cdots \sqcup \oC_\ell/ \sim.$$
  By Proposition \ref{p:bouquet}, the space
  $\cS_{\ell}$ is simply connected. For each $n$, we have the
  retraction
  $$\rho_{\ell,n} : \cS_{\ell} \to \oC_n$$
  which maps  $\oC_{n'}$ onto $a_{n',n}$, for all $n'\neq n$. 

  Collapsing $\oC_{\ell+1} \subset \cS_{\ell+1}$ to  a point 
  produces a continuous  map $\pi_{\ell+1}: \cS_{\ell+1} \to \cS_{\ell}$.
  We say that the inverse limit of $(\cS_{\ell}, \pi_\ell)$ 
  is the \emph{space $\cS$ associated  to the collection of scalings $\{A_{n,k} : n \ge m\}$}. 
  
  It is convenient to identify $\oC_n$ with
  the corresponding subset of $\cS_{\ell}$.
  In view of Lemma~\ref{l:decomposition},  $A_{n+1,k} \circ f_k \circ A_{n,k}^{-1}$ converges to a map $\varphi_{n,n+1} \in \Rat^*_d$.
  It is also convenient to regard
  $\tilde\varphi_{n,n+1}$
  as a map from $\oC_n$ onto $\oC_{n+1}$ and $A_{n,k}$ as a map from $\oC$ onto $\oC_n$. We will show that, under certain conditions, the maps $\tilde{\varphi}_{n,n+1}:\oC_n \to \oC_{n+1}$ produce
  a well  defined map from $\cS$  to $\cS$.

  

    \begin{theorem}[Structure Theorem]
    \label{p:structure}
    Let $\{f_k\}\subset\Ratd$ be a sequence such that $\llbracket f_k^{\circ n} ]\to\llbracket\varphi_n ]\in\llbracket \Rat_{d^n}^*]$ 
    for each $n\ge 1$ 
    and $[f_k]\to[g]\in\partial \rat_d$, as $k\to\infty$.    Let
       $\{A_{n,k}\}$ be a collection of post-scalings for $\{f_k\}$ having convergent changes of coordinates.

    Assume that there exists $m \ge 1$ such that $m, m+1,\dots,m+4$
    are fully ramified times. 
    Then $\{A_{n,k} :n \ge m\}$ is a collection of independent post-scalings and
    the  following  statements hold:
       \begin{enumerate}
       \item    The space $\cS$ associated to $\{A_{n,k}\}$
         is the linear concatenation of
         $\oC_{m}, \oC_{m+1}, \cdots$.
         That is, in suitable coordinates,
         employed for the rest of the statement,
         $\cS_{\ell}$ is
         obtained by identifying $0 \in \oC_n$ to $\infty \in \oC_{n+1}$,
         for all $m \le n < \ell$ (see Figure~\ref{fig:structure}). 
       \item  The map
         $$\begin{array}{cccl}
           \cF:&\cS&\to&\cS\\
               &z&\mapsto& \tilde\varphi_{n,n+1}(z) \in \oC_{n+1} \text{ if } z \in \oC_n,
          \end{array}
          $$ is well defined and continuous.
        \item For all $n > m$,
          $$\deg_{\infty} \tvarphi_{n,n+1}=\deg \tilde\varphi_{n,n+1}=\deg_0 \tvarphi_{n-1,n}.$$
       \item If {$n>m$ and} $n+1$  is
         a fully ramified  time, then $\tvarphi_{n,n+1}$ is a monomial of degree $d$.
      \item If  $n>m$ is not a fully
        ramified time, then  $z=\infty$ is the unique hole of $\varphi_{n,n+1}$.
       \end{enumerate}
  \end{theorem}

  \begin{figure}[h]
    \centering
    \centerline{\includegraphics[height=4cm]{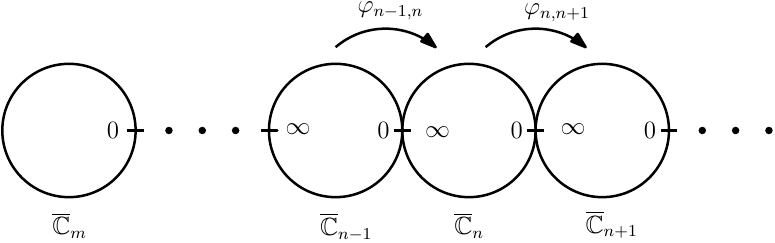}}
    \caption{Structure Theorem}
    \label{fig:structure}
  \end{figure}
  
  Continuity of $\cF$ is equivalent, once (1) is established, to $\tvarphi_{n,n+1}(\infty)=\infty$ and $\tvarphi_{n-1,n}(0)=0$ for $n>m$.
  Statement (3) yields that $\deg \tvarphi_{n,n+1}$ is a non-increasing function
  of $n>m$. Moreover, from   {statements (1)-(3)},  $\tvarphi_{n,n+1}$ is a polynomial for all $n>m$.

  The map $\cF: \cS \to \cS$ should be regarded as a limiting action
  of the sequence $\{f_k\}$.
  In general, let us assume that $\{A_{n,k}: n \ge m\}$ is a collection
  of independent post-scaling for a sequence $\{f_k\}$.
   For any $\ell \ge  m$, let 
   $$J_\ell:=\{n :  m  \le n  \le  \ell \}$$
    and consider the embedding
  $$\begin{array}{cccl}
           \cA_{J_\ell,k}:&\oC&\hookrightarrow&\oC^{J_\ell}\\
               &z&\mapsto& (A_{n,k}(z))_{n\in J_\ell}.
    \end{array}
  $$
  From Appendix~\ref{s:appendix}, 
 the space $\cS_\ell$ is naturally
  identified  with the  Hausdorff  limit of the embedded copies $\cA_{J_\ell,k}(\oC)$
  of $\oC$.
  Similarly, for $\ell+1$, the space $\cS_{\ell+1}$ is the limit
  of $\cA_{J_{\ell+1},k}(\oC)$. 
  Then, $f_k:\oC \to \oC$ acts on the embedded sphere as 
   $$
  \begin{array}{cccc}
    F_k :& \cA_{J_\ell,k}(\oC)& \to &\cA_{J_{\ell+1},k}(\oC)\\
         &    w & \mapsto & \cA_{J_{\ell+1},k} \circ f_k \circ \cA_{J_\ell,k}^{-1} (w).
  \end{array}
  $$
  Roughly speaking, Proposition~\ref{p:not-fully}
  allows us to count preimages under $F_k$. That is,
  consider  $z_0 \in \oC_n \subset \cS_{\ell}$ 
  and $w_0 \in \oC_{n+1} \subset \cS_{\ell+1}$.
Then, for sufficiently large $k$ and for any
$w \in \cA_{J_{\ell+1},k}(\oC)$ close to $\rho^{-1}_{\ell+1,n+1}(w_0)$,
there exists $\delta$ preimages of $w$, under $F_k$, close to
$\rho^{-1}_{\ell,n}(z_0)$ where 
$$\delta =
\begin{cases}
  d_{z_0} (\varphi_{n,n+1}), & \text{ if } \tilde{\varphi}_{n,n+1}(z_0)\neq w_0,\\
  d_{z_0} (\varphi_{n,n+1})+ \deg_{z_0} \tilde{\varphi}_{n,n+1}(z_0), & \text{ if } \tilde{\varphi}_{n,n+1}(z_0)= w_0.
\end{cases}
$$

See Proposition~\ref{p:not-fully} for a precise statement.

\begin{remark}
  In Berkovich space language, consider a type II point $x:=x_n$ that
  maps 
  onto  $y:=x_{n+1}$
  by a rational function $F$, and let $D_x(z_0), D_y(w_0)$ be directions
  at $x$ and $y$, respectively. Then the number of preimages in $D_x(z_0)$
  of a given point $w \in D_y(w_0)$ is prescribed by a well known formula
  that depends on the \emph{surplus multiplicity} and \emph{multiplicity} of the
  direction $D_x(z_0)$ (see \cite[Section 3]{Faber13}). The formula above is its Archimedean analogue.
   In this analogy, $x_n, x_{n+1}$
  correspond to $\oC_n, \oC_{n+1}$. 
  \end{remark}

  \subsection{Proof of the Structure Theorem}
  \label{s:proof-structure}

  The proof will be by induction on $\ell \ge m+4$.
In \S \ref{s:ih}, we reformulate
the assertions of the theorem in a manner suitable for induction.
In \S \ref{s:fr},  we
establish the inductive hypothesis for the initial case $\ell=m+4$ or more generally, for any $\ell =n_0 \ge m+4$ such that $n$ is a fully ramified time if
$m\le n\le n_0$.
Also, Theorem~\ref{C} is obtained as a byproduct.
In \S \ref{s:notfr}, the inductive step for $\ell \ge m+4$ is proven. 
We rely on  a series of lemmas, stated and established, under  the assumptions and notation of Theorem \ref{p:structure}.

\subsubsection{Inductive Hypothesis}
\label{s:ih}
Given $\ell \ge m+4$, consider the following assertions:
\begin{enumerate}
  \label{h:inductive}
       \item[(0.$\ell$)] For all $m \le n , n' \le \ell$, if $n \neq n'$, then
  $A_{(n,n')} \in \partial \Rat_1$. 
       \item[(1.$\ell$)]  $\cS_{\ell}$ is
         obtained by identifying $a_{n+1,n} \in \oC_n$ with $a_{n,n+1} \in \oC_{n+1}$ 
         for  $m \le n < \ell$. Moreover, $a_{n-1,n} \neq a_{n+1,n}$ for $m < n < \ell$.
       \item[(2.$\ell$)]
       $ \tilde\varphi_{n,n+1}(a_{n-1,n}) =  a_{n,n+1}$ if $m < n < \ell$, and   $\tilde\varphi_{n,n+1}(a_{n+1,n})  =  a_{n+2,n+1}$  if  $m \le n < \ell-1$.
       \item[(3.$\ell$)]
         $ \deg \tilde\varphi_{n,n+1}=\deg_{a_{n,n-1}} \tilde\varphi_{n-1,n}$ for  $m<n < \ell$.
       \item[(4.$\ell$)] If $m<n<\ell$ and $n+1$ is a fully ramified time, 
         then the local degree
         of $\varphi_{n,n+1}$ at  $a_{n\pm 1,n}$ is $d$.
        \item[(5.$\ell$)]  For  $m < n <\ell$, 
          $$d= d_{a_{n-1,n}} (\varphi_{n,n+1})+ \deg_{a_{n-1,n}}
          \tilde{\varphi}_{n,n+1}.$$
\end{enumerate}

Once the induction is implemented, Theorem~\ref{p:structure} is an immediate consequence.

\begin{proof}[Proof of Theorem~\ref{p:structure} assuming (0.$\ell$)--(5.$\ell$) for all $\ell >m$]
Clearly  (0.$\ell$) for all $\ell$ yields the independence of $\{A_{n,k}:n\ge m\}$. Moreover, (1.$\ell$) implies statement (1) of Theorem~\ref{p:structure}.
For statement (2),
we only need to check continuity at the intersection of consecutive
spheres, which is the content of (2.$\ell$). Provided we normalize $a_{n,n-1}=0 \in \oC_{n-1}$ and $a_{n-1,n} = \infty \in \oC_n$, for all $n$,
statement (3.$\ell$) for all $\ell$ is
the second equality in (3).
Statement  (4) is exactly (4.$\ell$) for all $\ell$.
For statement (5) and the remaining part of statement (3),
note that $d-\deg\tvarphi_{n,n+1}$ is the total number of holes
of $\varphi_{n,n+1}$, counted with depths. Therefore, 
$$d_{a_{n-1,n}} (\varphi_{n,n+1}) \le  d-\deg\tvarphi_{n,n+1}\le d- \deg_{a_{n-1,n}}
          \tilde{\varphi}_{n,n+1}.$$
          Thus, (5.$\ell$) for all $\ell$ yields equality
          between all these quantities.
          Equivalently,  ${a_{n-1,n}}$ is the unique hole of $\varphi_{n,n+1}$ and
          the local degree of $\tilde{\varphi}_{n,n+1}$
          at this point is $\deg\tvarphi_{n,n+1}$.  Statements
          (3) and (5) now follow.
     \end{proof}


\subsubsection{Proof of the Structure Theorem for fully ramified times}
\label{s:fr}
Choose $n_0\ge m+4$ and assume that $n$ is a fully ramified
time for all $n$ such that $m+4 \le n \le n_0$.
Our aim here is to prove (0.$\ell$) to (5.$\ell$) for $\ell=n_0$.

We start proving (0.$n_0$). In fact, it will be more convenient
to prove (0.$n_0+1$):
\begin{lemma}
  \label{lem:non}
If $m\le n < n' \le n_0+1$, then $A_{(n',n)} \in \partial \Rat_1$.
\end{lemma}

\begin{proof}
By contradiction, suppose that $A_{(n',n)}$ is non-degenerate. Then 
$$[f^{\circ n'-n}_k]=[ (A_{n,k} \circ A_{n',k}^{-1}) \circ (A_{n',k} \circ f^{\circ n'-n}_k \circ A_{n,k})] \to [A_{(n',n)} \circ \varphi_{n'-1,n'} \circ \cdots \circ \varphi_{n,n+1}] \in \rat_{d^{n'-n}},$$
  which contradicts that $[f_k] \to [g] \in \partial \rat_d$ since
  the $(n'-n)$-th iterate map is proper (see \cite[Proposition 4.1]{DeMarco07}).  
\end{proof}


Let $\cS'$ be the space associated to $\{A_{n,k}: m+1\le n\le n_0+1\}$.
According to Proposition~\ref{p:critical},
  $$\cF_{n_0}:\cS_{n_0} \to \cS',$$
  given by $$\cF_{n_0}(z) :=\varphi_{n,n+1}(z) \in \oC_{n+1},$$
for $z \in \oC_n$ and $n \in J_{n_0}$, 
is continuous. 
Moreover, there exists a finite set $\crit(\cF_{n_0}) \subset \cS_{n_0}$
  of \emph{critical points}, each with an assigned multiplicity, such
  that the total number of critical points is $2d-2$, counted with  multiplicities.
  Furthermore, $\crit(\cF_{n_0})$ has the following property: if $n \in J_{n_0}$ and
  $z \in \oC_n$,
  then the multiplicity of $z$ as a critical point of $\tilde\varphi_{n,n+1}$ is the number of elements of $\crit(\cF_{n_0})$ in $\rho_{{n_0},n}^{-1} (z)$, counted with multiplicities.

 

\begin{lemma}
  \label{l:order}
  There exists a total order $\prec$ in $J_{n_0}$
  with least element $m$ such that 
$$  \begin{array}{ccc}
    J_{n_0-1} &\to & J_{n_0} \\
    n &\mapsto & n+1
    \end{array}
    $$
    is a monotone map. Moreover, given  $n \neq n' \in J_{n_0}$, 
for  $\oC_{n}, \oC_{n'} \subset \cS_{n_0}$, the following hold: 
  \begin{enumerate}
  \item
     $\oC_n \cap \oC_{n'} \neq \emptyset$
if and only if  $n$ and $n'$ are $\prec$-consecutive. 
\item If $x \in \oC_{n}\cap \oC_{n'}$, then $\deg_x \varphi_{n,n+1} =d$.
\item If $n,n'$ are $\prec$-consecutive in $J_{n_0-1}$, then $n+1,n'+1$ are $\prec$-consecutive
  in $J_{n_0} \setminus \{m\}$. 
  \end{enumerate}
\end{lemma}

\begin{proof}

  To lighten notation, when possible,
  we omit sub-indices and write $\cS:=\cS_{n_0}$, $J=J_{n_0}$ and $\cF:=\cF_{n_0}$.
  Let us first show that given $n \in J$, there are
  at most two spheres, distinct from $\oC_n$, that intersect $\oC_n$.
  Suppose on the contrary that there are three distinct numbers $n_1, n_2, n_3 \neq n$ in $J$ such that
  the corresponding spheres have non-empty intersection with $\oC_n$ at
  $p_1, p_2, p_3$, respectively. Simple connectivity of $\cS$ yields
  that   $X_j:=\oC_{n_j} \setminus \{p_j\}$ 
  are pairwise disjoint for $j=1,2,3$, and hence 
  $\cT_j:=\rho_{n_0, n_j}^{-1} (X_j)$
  are also pairwise disjoint for $j=1,2,3$.
  Since $\varphi_{n_j,n_j+1}$ has degree $d$, its local degree at $p_j$ is at most $d$. Thus $X_j$ contains at least $d-1$ critical points of $\varphi_{n_j,n_j+1}$,
  counted with multiplicities. Therefore, $\cT_j$ contains at least $d-1$
  points of $\crit(\cF)$, counted with multiplicity, by Proposition \ref{p:critical}.
  Then $\crit(\cF)$ has at least $3d-3 > 2d -2$ points, counted with multiplicities, which contradicts Proposition \ref{p:critical}. 
  A similar argument shows that three distinct spheres do not share a common point.

  To define the order in $J$. Declare $m \prec m +1$. 
  Consider the graph $\cT$ with vertex set $J$
  and edges given by all pairs $\{n,n'\}$ such that $\oC_n \cap \oC_{n'} \neq
  \emptyset$ in $\cS$.  Since there are no triple intersection points and
  $\cS$ is simply connected, by Proposition \ref{p:bouquet}, we conclude that $\cT$ is a connected tree. Moreover, the previous paragraph also implies that  each vertex of $\cT$ has valence at most $2$. Thus $\cT$
  is isomorphic to an interval graph: it has exactly two valence $1$ vertices,
  say $a$ and $b$.
  Without loss of generality, $m+1$  belongs to
  the subtree connecting $m$ with $b$.
  Declare  $n \prec n'$
  if and only if $n$ belongs to  the connected component of $\cT \setminus \{n'\}$ containing $a$; equivalently, generic elements of
  $\oC_{n}$ and $\oC_a$ lie in the same connected component
  of $\cS \setminus \oC_{n'}$.

   Suppose that $n \prec n'$ are two consecutive elements of $J_{n_0}$.
   Denote by $x \in \cS$ the intersection point of $\oC_n$ and $\oC_{n'}$.
We claim that the
local degree of $\varphi_{n,n+1}$ at $x$  is $d$.
Indeed,  
the critical multiplicity of $\varphi_{n,n+1}$
at $x$ is  the number of elements of $\crit(\cF)$, counted with multiplicity, in
$\rho_{n_0,n}^{-1}(x) \supset \rho_{n_0,n}^{-1}(\oC_{n'}\setminus \{x\})$, which is at least $d-1$ since
$\deg \varphi_{n',n'+1} = d$.
Therefore, $\deg_x \varphi_{n,n+1}=d$.
A similar argument shows that $\deg_x \varphi_{n',n'+1} =d$.

Finally, if $n,n' \in J_{n_0-1}$ are $\prec$-consecutive,
then  $n +1$ and $n'+1$ are consecutive in $J\setminus\{m\}$,
 since $\cF (x)$ lies in the intersection of the corresponding spheres.
Moreover, the translation map $n \mapsto n+1$ is injective and, being consecutive preserving, must also be monotone.
\end{proof}

We conclude that if the translation by $1$ is $\prec$-monotone increasing, then 
$$m \prec m+1 \prec \cdots \prec n_0-1 \prec n_0.$$
That is, the order $\prec$ coincides with the standard order of the integers. 
Otherwise, if the translation by $1$ is $\prec$- monotone decreasing,
then either 
\begin{equation}\label{eq:1}
m \prec m+2 \prec \cdots  \prec m+3 \prec m+1, 
\end{equation}
or 
\begin{equation}\label{eq:2}
  \cdots\prec m+2 \prec m \prec m+1 \prec m+3 \prec \cdots. 
\end{equation}
However, studying the location of the Julia set $\cJ(f_k)$ of $f_k$ for sufficiently large $k$, we will obtain an obstruction for translation to be decreasing.
The location of the Julia set is controlled by the following general result: 

\begin{proposition}
  \label{p:julia}
  Consider $\{h_k\} \subset \Ratd$. Assume that $\{B_{n,k}:n=0,1 \}$ are a
  pair of independent post-scalings such that the associated space is obtained
  from the identification of $ a \in \oC_0$ with  $b \in \oC_{1}$.
  Suppose
  that the following hold:
  \begin{enumerate}
  \item $B_{1,k} \circ h_k \circ B^{-1}_{0,k}$ converges uniformly to $ \varphi \in \Ratd$.
  \item  $\deg_a \varphi = d$ and $\varphi(a) \neq b$.
  \end{enumerate}
  Given neighborhoods $W$ of $z=a$ and $V$ of $z=b$,
  for all $k$ sufficiently large, the following statements hold:
  \begin{enumerate}
  \item[(i)]    $\cJ(h_k) \subset B_{1,k}^{-1}(V)$
  and $\oC \setminus B_{1,k}^{-1}(V)$ is contained in a completely invariant
  attracting Fatou component of $h_k$.
\item[(ii)]
  There exists  a Jordan neighborhood $D'_k \subset W$ of $z=a$ such that
  $(U_k',U_k,h_k)$ is a degree $d$ polynomial like map, where
  $U'_k:=B_{0,k}^{-1}(\oC_0 \setminus \overline{D'_k})$.
  \end{enumerate}
\end{proposition}

\begin{proof}
  Changing  $B_{1,k}$ by $B \circ B_{1,k}$ for a suitable M\"obius transformation $B$, we may assume that $b=\infty$ and $\varphi(a) =0$.
  Similarly, we may assume that $a=0$.

  Let $H_k:= B_{0,k}\circ h_k \circ B^{-1}_{0,k}$ and
  $C_k:=B_{0,k}\circ B^{-1}_{1,k}$.
  Since $\infty \in \oC_1$ is identified with $0\in \oC_0$, it follows
  that $C_k$ converges to $C \in \partial \Rat_1$ such that $\tilde{C}=0$
  and the hole of $C$ is at $\infty$.
  Write $T_k:=B_{1,k} \circ h_k \circ B_{0,k}^{-1}$.
  
  Let $D$ be a small Jordan neighborhood of  $z=0$ ($=\varphi(a)$).
  In view of (2),
  $D':=\varphi^{-1}(D)$ 
  is a small neighborhood of $z=0$ ($=a$)
  which maps $d$-to-$1$ onto $D$, under $\varphi$.
  The uniform convergence of $T_k$ to $\varphi$ implies that
  $D'_k:=T_k^{-1}(D)$ is also a small Jordan neighborhood
  of $z=0$, say contained in $W$.
  Moreover, $H_k=C_k \circ T_k $ maps $D'_k$ onto
  $D_k:= C_k(D)$ with degree $d$.
  From the uniform convergence of $C_k$ converges to $0$ in
  $D$, we have $D_k \Subset D'_k$ for $k$ large.
  Then $H_k: \oC_0 \setminus {D'_k} \to \oC_0 \setminus {D}_k$
  is a degree $d$ polynomial like map which, after conjugacy by $B^{-1}_{0,k}$, becomes the desired polynomial like restriction of $h_k$. That is, statement (ii) holds.
  Note that $D_k'$ must be contained in a completely invariant attracting
  Fatou component of $H_k$.
  Since $C_k^{-1}$ converges uniformly to $\infty$ outside $D'_k$,
  it follows that $C_k^{-1} (\cJ(H_k)) = B_{1,k} (\cJ(h_k))$ converges, in the Hausdorff topology, to $\{\infty\}$. In particular, for $k$ large, $B_{1,k} (\cJ(h_k)) \subset V$. Similarly, $C^{-1}_{1,k}(D_k')$ contains $\oC\setminus V$ for
  $k$ large; therefore, $\oC \setminus B_{1,k}^{-1}(V)$ is contained in a completely
  invariant attracting Fatou component of $h_k$.
\end{proof}

We remark that the above polynomial like map is extracted from
a sequence $\{h_k\}$ with $2$ fully ramified times.
However, this extraction requires that assumption (2) is satisfied.
To prove Theorem~\ref{C}, we will show that assumption (2) always holds for sequences with $5$ consecutive fully ramified times.

Now we are ready to rule out monotone decreasing translations. Note that translation by $2$ is always monotone increasing from $J_{n_0-2}$
into $J_{n_0}$.

\begin{lemma}
  \label{l:increasing}
    The map $n \mapsto n +1$ is $\prec$-monotone increasing in $J_{n_0-1}$. 
\end{lemma}

\begin{proof}
  We proceed by contradiction and suppose
  that $n \mapsto n+1$ is $\prec$-orientation reversing.
In view of Lemma~\ref{l:order},  Proposition \ref{p:julia} applies to
$\{f^{\circ 2}_k\}$ for each one of the following situations:
\begin{enumerate}
\item[(i)] $B_{0,k}=A_{m,k}$, $B_{1,k} = A_{m+2,k}$, {$a=a_{m+2,m}$,
  $b = a_{m,m+2}$} and $\varphi = \varphi_{m+1,m+2}\circ \varphi_{m,m+1}$.
\item[(ii)] $B_{0,k}=A_{m+1,k}$, $B_{1,k} = A_{m+3,k}$, {$a=a_{m+3,m+1}$,
  $b = a_{m+1,m+3}$} and $\varphi = \varphi_{m+2,m+3}\circ \varphi_{m+1,m+2}$.
\end{enumerate}
Note that in both situations, by Lemma \ref{l:order}, $\deg_a \varphi =d^2$.
Moreover,  continuity of $\cF_{n_0}$ together
with the consecutive preserving property implies that
{$\varphi(a)=a_{m+4,m+2}\neq b$}  in (i), and
{$\varphi(a)=a_{m+5,m+3}\neq b$}  in (ii).

Suppose that  \eqref{eq:1} occurs. That is,
we have the concatenation of spheres:
$$\oC_{m}, \oC_{m+2}, \oC_{m+4}, \dots,  \oC_{m+5}, \oC_{m+3}, \oC_{m+1}.$$
Intuitively, Proposition \ref{p:julia} yields that
the Julia set $A_{m+2,k} (\cJ(f^{\circ 2}_k))$ is, in  $\oC_{m+2}$,
close to $\oC_m$, 
and  simultaneously, $A_{m+3,k} (\cJ(f^{\circ 2}_k))$ is, in $\oC_{m+3}$, close to $\oC_{m+1}$, which is impossible.
More precisely, let $W$ be a small neighborhood of
$a_{m,m+2}$ and $W'$ be a small neighborhood of $a_{m+1,m+3}$.
  Note that $(A_{m+3,k}\circ A_{m+2,k}^{-1})(W)$ converges to $a_{m+2,m+3} \neq a_{m+1,m+3}$ since the hole of $A_{(m+2,m+3)}$ is $a_{m+3,m+2} \neq a_{m,m+2}$.
Thus, for sufficiently large $k$,
  $$(A_{m+3,k}\circ A_{m+2,k}^{-1}) (W) \cap W' = \emptyset.$$
  Moreover, by Proposition~\ref{p:julia}, we have $\cJ(f_k^{\circ 2}) \subset A_{m+2,k}^{-1} (W) \cap A_{m+3,k}^{-1} (W')$. Therefore, 
  $$A_{m+3,k} (\cJ(f_k^{\circ 2})) \subset (A_{m+3,k}\circ A_{m+2,k}^{-1}) (W) \cap W' = \emptyset,$$
  which is a contradiction and \eqref{eq:1} cannot occur.

  Now suppose that  the order is as in \eqref{eq:2}. 
  Then we have the concatenation of spheres:
  $$\dots,\oC_{m+4}, \oC_{m+2}, \oC_{m},  \oC_{m+1}, \oC_{m+3},\oC_{m+5},\dots$$
  Intuitively, Proposition \ref{p:julia} yields that 
    the Julia set $\cJ(f^{\circ 2}_k)$ is between $\oC_{m}$ and  $\oC_{m+1}$, for $k$ large. We will find coordinate changes $A_{0,k}$ such that $A_{0,k}\circ f^{\circ 2}_k\circ A_{0,k}^{-1}$ converges to $z \mapsto z^{d^2}$ in $\Rat_{d^2}$,
    which contradicts the assumption that $[f_k]$ diverges in $\ratd$, since
    $\ratd \ni [f] \mapsto [f^{\circ 2}] \in \rat_{d^2}$ is a proper map.
  
    Consider a repelling fixed point $z_k\in\oC$ of $f_k$.
    By Proposition~\ref{p:julia}, $A_{m,k} (z_k) \to {a_{m+1,m}}$ and
    $A_{m+1,k}(z_k)\to {a_{m, m+1}}$.
  Choose $\alpha \in \oC_{m+4}$ and $\beta \in \oC_{m+5}$ such that
  $\alpha \neq a_{m+2,m+4}$ and $\beta \neq a_{m+3,m+5}$. 
  Let $A_{0,k}\in\Rat_1$ be defined by 
  \begin{enumerate}
  \item $A_{0,k}(z_k)=1$,
  \item $A_{0,k}\circ A_{m+4,k}^{-1}(\alpha)=0$,
  \item $A_{0,k}\circ A_{m+5,k}^{-1}(\beta)=\infty$.
  \end{enumerate}

  It is not difficult to check that
   $\{A_{n,k}: n=0, m\le n \le m+5\}$ is a collection of
  independent scalings and the associated space
  is the concatenation of spheres
  $$\oC_{m+4}, \oC_{m+2},\oC_m, \oC_0,\oC_{m+1},\oC_{m+3},\oC_{m+5}.$$
  Moreover, $a_{m+1,m} \in \oC_m$ is identified with $0 \in \oC_0$
  and $a_{m,m+1} \in \oC_{m+1}$ is identified with $\infty \in \oC_0$.

  Let $h_k:=A_{0,k} \circ f^{\circ 2}_k \circ A^{-1}_{0,k}$. 
  Since $h_k(1) =1$ for all $k$, to prove that $h_k(z)$ converges to $z^{d^2}$,
  it is sufficient to show that the poles of $h_k$ converge to $z=\infty$ and
  the zeros converge to $z=0$. 
  Observe that
  $T_k:=A_{m+5,k} \circ f^{\circ 2}_k \circ A^{-1}_{m+3,k}$ converges uniformly
  to a map $\varphi \in \Rat_{d^2}$.
  Moreover, $\varphi$
  maps  $a_{m+1,m+3}$  onto  $a_{m+3,m+5}$ with local degree $d^2$.
  Hence, the preimages of $\beta$, under $T_k$, are uniformly bounded away
  from $a_{m+1,m+3}$. Since
  $$h^{-1}_k(\infty)= (A_{0,k} \circ A^{-1}_{m+3,k}) \circ T^{-1}_k (\beta),$$
  and the hole of $A_{(m+3,0)}$ is $a_{m+1,m+3}$, it follows that all the poles
  of $h_k$ converge to $a_{m+3,0}=\infty \in \oC_0$.
  Similarly, all the zeros of $h_k$ converge to $0 \in \oC_0$.
\end{proof}

\begin{lemma}
  Statements (0.$n_0$)--(5.$n_0$) hold.
\end{lemma}

\begin{proof}
  As observed above, Lemma~\ref{lem:non} is (0.$n_0$).  By Lemmas~\ref{l:order} ~and \ref{l:increasing}, the space $\cS_{n_0}$ is the
  concatenation of spheres indexed by $J_{n_0}$ in
  the standard order of the integers. Therefore, statement (1.$n_0$) holds. Statement (2.$n_0$) follows from
  the continuity of $\cF_{n_0}$, furnished by Proposition~\ref{p:critical}.
  Statement (3.$n_0$) is trivially true and statement (4.$n_0$) is a direct consequence of Lemma~\ref{l:order} (2). Finally, statement (5.$n_0$) follows since
  $\varphi_{n,n+1}$ has no holes and the local degree at $a_{n-1,n}$ is $d$.
\end{proof}

We are now ready to deduce Theorem~\ref{C}.

\begin{proof}[Proof of Theorem~\ref{C}]
 
  Assume that $\{f_k\} \subset \Rat_d$, $\{A_k\}$
  and $\{B_k\}$ are as in the statement of the theorem.
  We proceed by contradiction. After passing to a
  subsequence we suppose that $f_k$ lacks of a degree $d$ polynomial
  like restriction for all $k$.
  Then $h_k:=A_k \circ f_k \circ A_k^{-1}$ also has no degree $d$ polynomial
  like restriction.
  Passing to a further subsequence, 
  $\llbracket A_k \circ f_k^{\circ n} \circ A_k^{-1}]$ converges to some $\llbracket \varphi_n ]$ 
  for all $n$. Moreover, $\varphi_5 \in \Rat_{d^5}$.
  From Lemma~\ref{l:decomposition}, the sequence $\{h_k\}$ has fully ramified
  times $n=1, \dots, 5$. Let $\{A_{n,k}\}$ be the collection of scalings
  such that $A_{n,k} \circ h_k^{\circ n} \to \varphi_n$.
  Then we are under the hypothesis of Theorem~\ref{p:structure}
  for $\{h_k\}$ and $m=1$. Therefore, statements (0.$5$)--(5.$5$) hold for $\{h_k\}$.
  In particular, $\cS_4$
  is the concatenation of $4$ spheres $\oC_1,\dots,\oC_4$.
  Moreover, $\varphi_{2,3}$ maps $\oC_2$ onto $\oC_3$ by a degree
  $d$ map. The intersection points of $\oC_2$ with $\oC_1$ and $\oC_3$ map,
  by $\varphi_{2,3}$,
  onto the (distinct) intersection points of $\oC_3$ with $\oC_2$ and $\oC_4$
  with local degree $d$. Thus, we are under the hypothesis
  of Proposition~\ref{p:julia}, for $\{A_{n,k}: n=1,2\}$.
  For $k$ large, we conclude the existence of
  a degree $d$ polynomial like restriction of
  $h_k$ and, therefore, of $f_k$.  
\end{proof}


 \subsubsection{Proof of the Structure Theorem for not fully ramified times}
\label{s:notfr} 
We continue with the induction.
Assume that (0.$\ell$)--(5.$\ell$) from \S \ref{h:inductive} hold
for some $\ell\ge n_0$.
The proofs of the corresponding statements for $\ell+1$
are organized in a series of lemmas below.

\smallskip
We employ the definitions and results from Appendix~\ref{s:appendix}.
Let $J := \{ n : m \le n \le \ell\}$ and
 consider the collection of independent scalings $\mathbf{A}:=
 \{A_{n,k}:  n \in J \}$ indexed by $J$. The space associated
 to $\mathbf{A}$ is $\cS_{\ell}$, the concatenation of the spheres $\oC_m,\dots,\oC_\ell$.
 Denote by $\cA_k$ the corresponding embedding of
 $\oC$ into $\oC^{J}$ and by
 $\rho_{\ell,j}$ the retraction of $\cS_{\ell}$
 onto $\oC_j$. For all $n \neq j$,
 note that  $\rho_{\ell,j}(\oC_n) = \{a_{n,j}\}\subset \oC_j$.

 \begin{lemma}
   \label{l:indep}
   $A_{(\ell+1,j)}$ is degenerate for all $j \in J$.
 \end{lemma}
 \begin{proof}
   By contradiction, suppose $A_{(\ell+1,j)}$ is non-degenerate.
 Without loss of generality, we may assume that
 $A_{\ell+1,k} = A_{j,k}$, for all $k$.
 In the language of Appendix~\ref{s:appendix},
the sequence $\{f_k\}$ maps  $\mathbf{A}$ to $\mathbf{A}$
 by $\tau:J\to J$ where $\tau(n) = n+1$ for $n < \ell$ and $\tau(\ell)=j$.
 Consider $G_k : \cA_k(\oC) \to \cA_k(\oC)$ defined by
 $$G_k := \cA_k \circ f_k \circ \cA_k^{-1}.$$
 In view of the inductive hypothesis (5.$\ell$) and Proposition~\ref{p:not-fully} applied to $j-1$,
 every point in $\cA_k (\oC)$, contained in a small neighborhood $W_0$ of 
  $\rho_{\ell,j}^{-1}(a_{j-1,j})=
 \oC_{m} \cup \cdots \cup \oC_{j-1}$, has $d$ preimages
 under $G_k$, close to $\rho_{\ell,j-1}^{-1} (a_{j-2,j-1})=\oC_m \cup
 \dots \cup \oC_{j-2}$, which is bounded away from $\oC_{\ell}$.
 But $W_0$ contains points $w_0$ in  $\oC_{j}$
 close to $a_{j-1,j}$ with a preimage $z_0$ in $\oC_{\ell}$, under $\tilde\varphi_{\ell,j}=
 \tilde\varphi_{\ell,\ell+1}$.
 It follows that points in $\cA_k (\oC)$, close to  $w_0$,
 simultaneously have a $G_k$-preimage close to  $z_0 \in \oC_{\ell}$ and
 $d$ preimages close to $\rho_{\ell,j-1}^{-1} (a_{j-2,j-1})$, which  is impossible.
\end{proof}

Let $J':=\{ j : m \le n \le \ell +1 \}$. From the previous lemma,
$\mathbf{B} :=
\{ A_{n,k} : n \in J'\}$ is a collection of independent scalings indexed by
$J'$. 
It follows that $\{f_k\}$ maps $\mathbf{A}$ to
$\mathbf{B}$ by $\tau(n)=n+1$.
For all $j \in J'$,
 recall that $\rho_{\ell+1,j}$ is the retraction from  $\cS_{\ell+1}$, the space associated to $\mathbf{B}$,
 onto $\oC_j$.
 Let $\cB_k$ be the corresponding embedding of $\oC$ into $\oC^{J'}$
 and $$F_k:= \cB_k \circ f_k \circ \cA_k^{-1}.$$

 \begin{lemma}
   \label{l:concatenation}
  $a_{\ell+1,\ell} \neq a_{\ell-1,\ell} \in \oC_{\ell}$. Moreover,
  $\cS_{\ell+1}$ is obtained from the union of $\oC_m,\dots,\oC_{\ell+1}$
  by identifying $a_{n+1,n}$ with $a_{n,n+1}$ for $m\le n < \ell+1$.
\end{lemma}

\begin{proof}
 By contradiction, suppose that $a_{\ell+1,\ell} = a_{\ell-1,\ell}\in \oC_{\ell}$.
 Then $\oC_{\ell-1}, \oC_{\ell}, \oC_{\ell+1}$ intersect at 
  $a_{\ell,\ell-1} \in \oC_{\ell-1}$, which is identified with $a_{\ell-1,\ell} \in \oC_\ell$.
Thus $\rho_{\ell+1,\ell}^{-1} (a_{\ell-1,\ell})$ contains $\oC_{\ell+1}$.
Any point in $\cB_k(\oC)$ close to a generic point $w_0$ of $\oC_{\ell+1}$ has a $F_k$-preimage in $\cA_k(\oC)$ close to $\oC_{\ell} \subset \cS_\ell$.
Proposition ~\ref{p:not-fully} and (5.$\ell$) yield a contradiction, since the 
 $d$ preimages of $w_0$ are close to $\rho_{\ell, \ell-1}^{-1} (a_{\ell-2,\ell-1})=
 \oC_{m} \cup \cdots \cup \oC_{\ell-2}$, which is bounded away from
 $\oC_{\ell}$.
 Thus, $a_{\ell+1,\ell} \neq a_{\ell-1,\ell}\in \oC_{\ell}$

We conclude that in $\cS_{\ell+1}$ the sphere
$\oC_{\ell+1}$ has non-empty intersection only with $\oC_\ell$; for otherwise, if $\oC_{\ell+1} \cap \oC_n \neq \emptyset$ for some $n \le \ell-1$, then $a_{\ell+1,\ell} = a_{\ell-1,\ell}$. Therefore, $\cS_{\ell+1}$
is the claimed concatenation of spheres.
\end{proof}

According to the next lemma, in coordinates  for $\oC_\ell$ and $\oC_{\ell+1}$
where  $a_{\ell-1,\ell}$ and
$a_{\ell,\ell+1}$ are at $\infty$, the map $\tilde{\varphi}_{\ell,\ell+1}$
becomes a polynomial. Moreover,
$a_{\ell-1,\ell}$ is the unique hole of $\varphi_{\ell,
  \ell+1}$, if $\varphi_{\ell,\ell+1}\in \partial\Rat_d$. 
In particular, (5.$\ell+1$) and the first assertion of (2.$\ell+1$) 
hold.

\begin{lemma}
  \label{l:poly}
  Let $a :=a_{\ell-1,\ell}$ and
  $a':= a_{\ell,\ell+1}$ and $\varphi:={\varphi}_{\ell,\ell+1}$.
  Then $\tilde{\varphi}^{-1}(a')
= a$ and $$d= \deg_a \tilde{\varphi}+ d_a (\varphi).$$
\end{lemma}

 \begin{proof} 
 Let $z$ be such that $\tilde{\varphi}(z)=a'$.
 Since 
 $\rho^{-1}_{\ell+1,\ell+1}(a')$
 contains $\oC_n$ for all $n \le \ell$,
 in particular, it contains $\oC_{\ell-1}$. If $w_k \in \cB_{k}(\oC)$
 is sufficiently close to a generic point of 
 $\oC_{\ell-1}$, then it has a preimage close to $\rho^{-1}_{\ell,\ell}(z)$.
 However, all the $F_k$-preimages of $w_k$ are close to
 of $\oC_m \cup \dots \cup \oC_{\ell-2}$, by (5.$\ell$) and  Proposition~\ref{p:not-fully}. It follows that $\rho^{-1}_{\ell,\ell}(z)$ is contained in  $\oC_m \cup \dots \cup \oC_{\ell-2}$, whose $\rho_{\ell,\ell}$-image is $\{a\}$. Thus
 $z=a$ is the unique $\tilde{\varphi}_{\ell,\ell+1}$-preimage of $a'$. Moreover, since all the $F_k$-preimages of $w_k$ are close to
 $\rho^{-1}_{\ell,\ell} (a)$, by Proposition~\ref{p:not-fully}, the second assertion holds. 
\end{proof}

Now we prove the second assertion of (2.$\ell+1$):

\begin{lemma}
  \label{l:zero}
  Let $b:=a_{\ell, \ell-1}$, $b':=a_{\ell+1,\ell}$ and
  $\varphi:={\varphi}_{\ell-1,\ell}$. Then $\tilde{\varphi} (b) = b'$. 
\end{lemma}

\begin{proof}
  By contradiction, suppose that $\tvarphi(b) \neq b'$.
  Then $\tvarphi^{-1}(b')$ consist of finitely many points in $\oC_{\ell-1}$
  bounded away from the rest of the spheres, by (1.$\ell$) and (2.$\ell$). Thus, $\rho_{\ell,\ell-1}^{-1}(z)=\{z\}$
for all $z \in \tvarphi^{-1}(b')$.
Therefore, the $F_k$-preimages of  points in $\cB_k(\oC)$ close to a
generic point of $\oC_{\ell+1}=\rho_{\ell+1,\ell}^{-1}(b')$ are close to the finite set $\tvarphi^{-1}(b') \subset \oC_{\ell-1}$ far from $\oC_\ell$. This
is impossible since $\tilde{\varphi}$ maps $\oC_{\ell}$ onto $\oC_{\ell+1}$.
\end{proof}

We also have the following result on degrees, which gives (3.$\ell+1$):
\begin{lemma}
  \label{l:degree}
 Let $b:=a_{\ell, \ell-1}$. Then $\deg \tilde{\varphi}_{\ell-1,\ell} \ge \deg_b \tilde{\varphi}_{\ell-1,\ell} = \deg \tvarphi_{\ell,\ell+1}$.
\end{lemma}

\begin{proof}
  Let $b':=a_{\ell+1,\ell}$. Note that $b$ is not a hole
  of $\tvarphi_{\ell-1,\ell}$ by (5.$\ell$).
  Let $w_0$ be a generic point of $\oC_{\ell+1}=\rho^{-1}_{\ell+1,\ell}(b')$.
  There are exactly  $\deg_b \tvarphi_{\ell-1,\ell}$ preimages, under $F_k$, 
  of every point in $\cB_k(\oC)$ close to $w_0$, which are close
  to $\oC_\ell = \rho^{-1}_{\ell,\ell-1}(b)$.
  But there are $\deg \tvarphi_{\ell,\ell+1}$ preimages of $w_0$ in $\oC_\ell$,
  so $\deg \tvarphi_{\ell,\ell+1}= \deg_b \tvarphi_{\ell-1,\ell} \le \deg \varphi_{\ell-1,\ell}.$
\end{proof}

We now conclude the induction by certifying that (0.$\ell+1$)--(5.$\ell+1$)
follow from the above lemmas.
Indeed, (0.$\ell+1$) is a consequence of Lemma~\ref{l:indep}, (1.$\ell+1$) is
Lemma~\ref{l:concatenation}.
Lemma~\ref{l:poly} implies the first assertion of (2.$\ell+1$) together with (5.$\ell+1$).
The second assertion of (2.$\ell+1$) is obtained
from Lemma~\ref{l:zero}. Statement (3.$\ell+1$)
is contained in Lemma~\ref{l:degree}.
Suppose that $n_0+1=\ell$ is not a fully ramified
time. Then (4.$\ell+1$) reduces to (4.$n_0$), which
we have already proven (Lemma~\ref{l:order} (2)).
This finishes the proof of Theorem~\ref{p:structure}.

\subsection{Measures and depths}
Consider a sequence $\{f_k\}$ with a collection of post-scaling $\{A_{n,k}\}$
for which the Structure Theorem~\ref{p:structure} applies.
Proposition~\ref{p:julia} says that the Julia set $\cJ(f_k)$, viewed
in $\oC_n$ for $n>m$, converges to $a_{n-1,n}$.
Measure-theoretically we have the following consequence:

\begin{corollary}\label{coro:main}
  Under the assumption and notation of Theorem \ref{p:structure},
  for any $n >m$, let $A_n:=\lim A_{n,k}$ and $a_n:= \tilde{A}_n$.
  Then the following hold:
        \begin{enumerate}
       \item  $a_n=a_{n-1,n}$.
       \item $A_{n,k} (\cJ(f_k))$ converges in the Hausdorff topology on compact sets to $\{a_n\}$. In particular, 
         $$(A_{n,k})_* \mu_{f_k} \to \delta_{a_n}.$$
              \end{enumerate}
\end{corollary}
\begin{proof}
  Consider $n > m$. By Theorem~\ref{p:structure},
  $$T_k:=A_{n+m,k} \circ f_k^{\circ m} \circ A_{n,k}^{-1}$$
  converges to a map $\varphi \in \Rat^*_{d^m}$ with a unique hole at
  $a_{n-1,n}=\infty$. If $a_n \neq a_{n-1,n}$, then
  $A_{n+m,k} \circ f_k^{\circ m} = T_k  \circ A_{n,k}$
  converges to $\tilde{\varphi}(a_n)$ for generic $z \in \oC$.
  By Theorem~\ref{p:structure} (4), the unique preimage of $a_{m,n+m}=\infty$ under the polynomial $\tilde{\varphi}$ is
  $a_{n-1,n}=\infty$. Thus, $\tilde{\varphi} (a_n) \neq \infty$.
  However, generically $A_{m,k}\circ f_k^{\circ m}$ converges to a non-constant
  map $\tilde{\varphi}_m$.  We conclude that, generically,  
$$(A_{m+n,k}\circ A_{m,k}^{-1})\circ (A_{m,k}\circ f_k^{\circ m})(z)\to a_{m,m+n}=\infty.$$
This is a contradiction. Hence statement (1) holds. 

For statement (2), consider a small neighborhood $W$ of $a_n$. By statement (1), we have that $a_n=a_{n-1,n}$ and, by Proposition \ref{p:julia} and Theorem \ref{p:structure}, we conclude that $W$ contains
  $A_{n,k} (\cJ(f_k))$ for all  sufficiently large $k$. The conclusion follows since  $A_{n,k} (\cJ(f_k))$ is the support of $(A_{n,k})_* \mu_{f_k}$.
\end{proof}

Fully Ramified Sequences have well-behaved limits in $\wRatd$: 
\begin{corollary}\label{coro:main1}
Under the assumption and notation in Theorem \ref{p:structure}, if $f_k^{\circ n}$ converges to $g_n \in  \oRat_{d^n}$ for all $n$ (i.e. $\{ f_k \}$ converges  to $(g_n)$ in $\wRatd$), then for any given $n>m$  the following hold:
\begin{enumerate}
\item $\eta_{g_n} = \eta_{g_{m+1}}$; 
\item $\tilde g_n=\tilde g_{m+1}$ is a constant in $\oC$; and
\item for any $z\in\oC$, 
         $$\frac{d_z(g_n)}{d^n}=\frac{d_z(g_{m+1})}{d^{m+1}}.$$
    \end{enumerate}  
\end{corollary}
\begin{proof}
  Consider $n > m$.
  Recall that, in convenient coordinates, we put $a_{n-1,n}$ at $\infty$.
  Let us first prove statement (2).
  The hole $h$ of $A_n = \lim A_{n,k}$ is $\tilde g_n$.
  By Corollary \ref{coro:main} (1), generically, $A_{n,k}$ converges to $a_n = a_{n-1,n}=\infty \in \oC_n$.
  Off $a_n$, we have that $A_{n,k}^{-1}$ converges to $h$.
  Since  $A_n \circ A_{n+1}^{-1}$ converges generically to $a_{n+1,n} \neq a_n$,
  it follows that
  $A_{n+1,k}^{-1} = A_{n,k}^{-1} \circ (A_n \circ A_{n+1}^{-1})$
  also converges to $h$. Thus the hole of $A_{n+1}$ is also $h$ and
  $\tilde{g}_{n+1} =h$.

  We now show  statement (3), which immediately implies statement (1). It is sufficient to
show that $d_z(g_{n+1})=d_z(g_n)\cdot d$ for all $z \in \oC$.  By Lemma~\ref{lem:s}, if $\tvarphi_{m+1} (z) \neq a_{m, m+1}$, then
 $d_z (g_n) = d_z (\varphi_n)$; and 
if otherwise, $d_z (g_n) = d_z (\varphi_n)+ \deg_z \tvarphi_n$. To compute $d_z(g_{n+1})$, we consider two cases according to whether $\tvarphi_{m+1} (z) = a_{m,m+1}$ or not, equivalently, $\tvarphi_n (z) = a_{n-1,n} =\infty \in \oC_n$ or not since all transition maps are polynomials (Theorem ~\ref{p:structure} (4)).

    Suppose that $\tvarphi_{m+1} (z) \neq a_{m,m+1}$.
  That is, $\tvarphi_n(z) \neq a_{n-1,n}=a_n$.
  Then $\tvarphi_n(z)$ is not a hole of $\varphi_{n,n+1}$ and $\tvarphi_{n+1} (z) \neq a_{n+1}$. Therefore, using
  the formula \eqref{eq:depth-composition},  
  we obtain that 
  $$d_z(g_{n+1}) = d_z (\varphi_{n+1}) = d_z (\varphi_{n}) \cdot d = d_z(g_n) \cdot d.$$

  Suppose that $\tvarphi_{m+1} (z) = a_{m,m+1}$.
  That is, $\varphi_n(z) = a_{n-1,n}=a_n$. Then $\varphi_n(z)$ is the hole
  of $\varphi_{n,n+1}$, which has depth $d - \deg \tvarphi_{n,n+1}$
  and the local degree at $a_n$ is $\deg \tvarphi_{n,n+1}$. It follows that 
  \begin{eqnarray*}
    d_z(g_{n+1}) & = & d_z (\varphi_{n+1}) + \deg_z \tvarphi_{n+1}\\
                 &=& d_z(\varphi_n) \cdot d + \deg_z \tvarphi_{n} \cdot d_{a_n} (\varphi_{n,n+1}) +  \deg_z \tvarphi_{n} \cdot \deg_{a_n} \tvarphi_{n,n+1}\\
                 &=& d_z(\varphi_n) \cdot d + \deg_z \tvarphi_{n} \cdot d\\
    &=& d_z(g_{n+1}) \cdot d.
  \end{eqnarray*}

\end{proof}

\section{Iterations and measures for degenerate maps}
\label{s:Aproof}

In this section, we prove Theorem \ref{A}.
Given $\mathbf{g} = (g_n) \in \wRatd$, we aim at proving
the existence of a probability measure $\mu_{\mathbf{g}}$, continuously depending
on $\mathbf{g}$, such that 
$$\dfrac{1}{d^n}(g_n)^* \mu \to \mu_{\mathbf{g}},$$
for a generic probability measure $\mu$.
To this end, we consider
a sequence $\{f_k\}\subset \Ratd$ converging to $\mathbf{g}$.
First we make some assumptions on the sequence $\{f_k\}$ and
prove Theorem~\ref{A} under these assumptions. Later, using compactness of $\wRatd$
and $M^1(\oC)$, we deduce Theorem~\ref{A} in full generality.

We may simultaneously consider the limits of $f_k^{\circ n}$ and of $\llbracket f_k^{\circ n} ]$; to record these limits, let us introduce the following definition: 
\begin{definition}\label{def:total}
A sequence $\{f_k\}\subset\Ratd$ is  \emph{totally convergent} if 
for any $n\ge 1$, there exist $g_n \in  \oRat_{d^n}$  and $\varphi_n\in\Rat_{d^n}^*$  such that the following hold:
\begin{enumerate}
\item $f_k^{\circ n}$ converges to $g_n \in  \oRat_{d^n}$ (i.e. $\{ f_k \}$ converges  to $(g_n)$ in $\wRatd$), and 
\item $\llbracket f_k^{\circ n} ]$ converges to $\llbracket \varphi_n ]\in\llbracket \Rat_{d^n}^*]$.
\end{enumerate}
We say that $\{f_k\}$ \emph{totally converges to} $(g_n,\varphi_n)$.
\end{definition}

To prove Theorem~\ref{A}, we assume that  $\{f_k\} \subset \Ratd$   totally converges to $(g_n, \varphi_n)$, as $k \to \infty$ (see Definition \ref{def:total}).
We also assume that the conjugacy classes $\{ [f_k] \} \subset \oratd$
converge in moduli space to some $[g] \in \oratd$.
Then  we consider the following cases:
\begin{enumerate}
\item $\deg \tilde{\varphi}_n = o(d^n)$.
\item $\deg \tilde{\varphi}_n \neq o(d^n)$ and
  \begin{enumerate}
  \item $[g] \in \ratd$ or,
  \item $[g] \in \partial \ratd$.
  \end{enumerate}
\end{enumerate}
In each case, we show that a measure $\mu_{\mathbf{g}}$ with the desired properties  exists and that $\mu_{f_k}$ converges to 
$\mu_{\mathbf{g}}$. After discussing some preliminary facts in \S \ref{s:depthcon},
 case (1) is considered in \S \ref{s:small}, case (2a) in \S \ref{s:potential}, and case (2b) in
\S \ref{s:ramified}. Finally, in \S \ref{s:A}, we assemble 
these results to produce a proof of Theorem~\ref{A}.

\subsection{Convergence of depth measures}
\label{s:depthcon}

  Recall from Definition \ref{def:depthmeasure} that the depth measure of a map $g \in \partial\Ratd$ is
  $$\eta_g := \sum_{h \in \hole(g)} d_h(g) \delta_h.$$
  
\begin{lemma}\label{lem:depthmeasure}
  Let  $\{f_k\} \subset \Ratd$ be a sequence such that $\llbracket f_k^{\circ n} ]\to\llbracket\varphi_n ]$ in $\llbracket \Rat_{d^n}^*]$ for each $n\ge 1$, as $k \to \infty$.
Then there exists a purely atomic measure $\eta$ such that, as $n\to\infty$, 
$$\frac{1}{d^n} \eta_{\varphi_n}\to \eta.$$
\end{lemma}

\begin{proof}
Apply  Lemma \ref{l:decomposition} to write $\varphi_{n+1} = \varphi_{n,n+1} \circ  \varphi_n$ where $\deg \tilde{\varphi}_{n,n+1} \ge 1$. 
  For all $h \in \oC$, in view of \eqref{eq:depth-composition}, we have 
   $$d_h (\varphi_{n+1}) = d_h (\varphi_n) \cdot d + \deg_{h} \tilde\varphi_n \cdot d_{\tilde\varphi_n(h)}(\varphi_{n,n+1}).$$
  Division by $d^{n+1}$ yields that
  $$\frac{\eta_{\varphi_n} (\{h\})}{d^n} = \dfrac{d_h(\varphi_n)}{d^n}$$
  is non-decreasing.
  Let
  $\eta$ be the purely atomic measure defined by
  $$\eta(\{h\}) := \lim_n\frac{\eta_{\varphi_n} (\{h\})}{d^n}.$$
  It follows that $\eta_{\varphi_n}/d^n \to \eta$.
\end{proof}

\subsection{Small degree growth}
\label{s:small}
Let $\{ f_k\}\subset\Ratd$ be a sequence totally converging
to $(g_n,\varphi_n)$. The growth of $\deg \tilde{\varphi}_n$ plays a crucial in our argument. In the case of small degree growth, that is, $\deg \tilde{\varphi}_n = o(d^n)$, we have the following convergence of measures.

\begin{proposition}\label{p:smalldegree}
  Consider a sequence
  $\{f_k\} \subset \Ratd$   totally converging to $(g_n, \varphi_n)$. 
  Suppose $\mu$ is a non-exceptional probability measure for $(g_n)$.
  Let $\eta := \lim (\eta_{\varphi_n}/d^n)$.
  Assume that
  $$\deg \tilde{\varphi}_n = o(d^n).$$
  Then
$$ \eta = \lim_{n\to\infty} \dfrac{1}{d^n} (g_n)^* \mu = \lim_{k\to\infty}  \mu_{f_k}.$$ 
\end{proposition}

\begin{proof}
  Let $\nu$ be an accumulation point of $\mu_{f_k}$.
  Consider $A_{n,k} \in \Rat_1$ such that $A_{n,k} \circ f_k^{\circ n} \to \varphi_n$, as $k \to \infty$, for
  all $n$. 
  Passing to a subsequence of $\{f_k\}$, we may assume that
  $\mu_{f_k}$ converges to $ \nu$ and 
   $(A_{n,k})_* \mu_{f_k}$ converges to some measure $\nu_n$ for all $n$. 
  By Proposition \ref{prop:mea-CE},
  $$\nu = \dfrac{1}{d^n} (\varphi_n)^* \nu_n = \frac{1}{d^n}\eta_{\varphi_n} + \dfrac{1}{d^n}(\tilde{\varphi}_n)^*\nu_n \to \eta,$$
  since $(\tilde{\varphi}_n)^* \nu_n (\oC) = \deg \tilde{\varphi}_n = o(d^n)$.
  Therefore,  the measures $\mu_{f_k}$ converge to $\eta$.

  To show the measure $(g_n)^* \mu/d^n$ converges to $\eta$, we now prove that the measures $(g_n)^*\mu$ and $\eta_{\varphi_n}$ differ
  in at most $o(d^n)$. Although the conclusion is the same, there
  are two cases, according to whether  $\tilde{g}_n$ is constant or not.
  
  For all $n$ such that $\tilde{g}_n$ is constant, from Lemma \ref{lem:s},
  the sequence $\{A_{n,k} \}$ converges to some $A_n \in \partial \Rat_1$, as $k \to \infty$, and 
  \begin{eqnarray*}
    (g_n)^*\mu =  \eta_{g_n} 
             = \eta_{\varphi_n} +  \sum_{z \in \tilde{\varphi}_n^{-1}(\tilde{A}_n)} \delta_z
              =  \eta_{\varphi_n} + \varepsilon_n,
  \end{eqnarray*}
  where $\varepsilon_n (\oC) = \deg \tilde{\varphi}_n$.
  
  If $\tilde{g}_n$ is not constant, then 
  \begin{eqnarray*}
    (g_n)^* \mu  =   \eta_{g_n} + (\tilde{g}_n)^* \mu
               =  \eta_{\varphi_n} + (\tilde{g}_n)^* \mu
              =  \eta_{\varphi_n} + \varepsilon'_n,
  \end{eqnarray*}
where $\varepsilon'_n (\oC) = \deg \tilde{g}_n$.
  
  For all $n \ge 1$,
  we have that $\deg \tilde{g}_n \le \deg \tilde{\varphi}_n = o (d^n)$.
  Therefore, both $\varepsilon_n(\oC)$  and $\varepsilon'_n(\oC)$  are $o(d^n)$, and hence 
  $$\lim_{n\to\infty} \dfrac{1}{d^n}(g_n)^* \mu = \lim_{n\to\infty} \dfrac{1}{d^n} \eta_{\varphi_n} =\eta.$$
  This completes the proof. 
\end{proof}

We will deal with sequences such that $\deg \tilde{\varphi}_n \neq o(d^n)$ in
the next two subsections, but before, let us record the following observation:

\begin{lemma}
  \label{l:constant}
  Let $\{ f_k\}$ be a sequence totally converging to $(g_n, \varphi_n)$
  and $g_1 \in \partial \Ratd$.
  Suppose that $\deg \tilde{\varphi}_n \neq o(d^n)$.
  Then  $$0<\dfrac{1}{d^n} \deg \tilde{\varphi}_n\le 1$$
  is eventually constant (independent of $n$).
  Moreover, for all $n$ 
  $$\deg \tilde{g}_n =0.$$
\end{lemma}

\begin{proof}
  Let $\varphi_{n,n+1} \in \Rat^*_d$ be such that $\varphi_{n+1}=\varphi_{n,n+1} \circ \varphi_n$, as in Lemma \ref{l:decomposition}.
  Since $\deg \tilde{\varphi}_{n+1} = \deg\tilde{\varphi}_{n,n+1} \cdot \deg \tilde{\varphi}_{n}$ and $\deg \tilde{\varphi}_n \neq o(d^n)$,
  it follows that $\deg \varphi_{n,n+1} = d$ for all $n$ sufficiently large.
  Therefore, $\deg \tilde{\varphi}_n/d^n$ is eventually constant in $(0,1]$.

  We now prove that $\deg \tilde{g}_n =0$, by contradiction.
  First observe that $g_n \in \partial \Rat_{d^n}$ because $g_1 \in \partial \Ratd$ and iteration is a proper map.
  Suppose that, for some $n_0$, we have $\deg \tilde{g}_{n_0} \ge 1$. Then $\deg \tilde{\varphi}_{\ell n_0}=\deg \tilde{g}_{\ell n_0} = \deg \tilde{g}_{n_0}^{\circ \ell} = o(d^{n_0 \ell})$, as $\ell\to\infty$, 
  since $\deg \tilde{g}_{n_0} < d^{n_0}$.
  Therefore, $\deg \tilde{\varphi}_{\ell n}/d^{n\ell}$ is not eventually constant, as $\ell \to \infty$, which is a contradiction.
\end{proof}

\subsection{Potential good reduction}
\label{s:potential}
Here we consider the case in  which $\{ f_k\}$ totally converges
to $(g_n,\varphi_n)$ with $\deg \tilde{\varphi}_n \neq o(d^n)$ and $\{ [f_k]\} \subset \ratd$ converges
in $\ratd$.  In order to state the hypothesis that are really used, we
consider a slightly more general case:

\begin{proposition}
  \label{p:potential}
  Let $\{ f_k\}\subset \Ratd$ be a sequence converging to
  $g \in \partial \Ratd$.
  Assume that 
  $\{[f_k]\}$ converges in moduli space $\ratd$.
  Then $g$ has a unique hole, say $h$.
  Moreover, if $f_k^{\circ n} \to g_n$, as $k \to \infty$ and $\tilde{g}_n$
  is constant, for all $n$,
  then $$\lim_{k\to\infty} \mu_{f_k}=\delta_h = \lim_{n\to\infty}\dfrac{1}{d^n}(g_n)^* \mu$$
  for all probability measures $\mu$ that are non-exceptional for $(g_n)$. 
\end{proposition}

\begin{proof}
  Let  $\{ L_k \} \subset \Rat_1$ be such that
$\psi_k:=L_k \circ f_k \circ L^{-1}_k$ converges (uniformly) to some $ \psi \in \Ratd$, as $k\to\infty$. 
  We may assume that $\{L_k\}$ converges to a degenerate M\"obius transformation with a hole at $h$ and with constant reduction $a$.
  By contradiction we prove that $h$ is the unique hole of $g$.
  Suppose that $h' \neq h$ is a hole of $g$. 
  Consider a small neighborhood $U$ of $h'$.
  Pick a generic point
  $z_0\in\oC$ such that $z_0\neq \psi(a)$. For sufficiently large $k$,
  from Lemma \ref{lem:pre-number} (1), we deduce that
  $z_0 \in L_k \circ f_k(U)$; therefore,
  $z_0 \in \psi_k \circ L_k(U)$.
  The uniform convergence, in $U$, of $\{L_k\}$  to
  the constant $a$ implies that $z_0 = \psi(a)$, which contradicts the choice of
  $z_0$. Thus the first assertion holds. 

  We now prove the second assertion. Since
  the delta measure $\mu_\psi$ is non-exceptional for $L=\lim L_k$  and
  $\mu_{\psi_k} \to \mu_\psi$, we conclude that $\mu_{f_k} = (L_k)^* \mu_{\psi_k} \to \delta_h$, as $k\to\infty$, by Proposition~\ref{p:pullback-continuity}.
  To establish $\delta_h = (g_n)^* \mu/d^n$, 
  note that $L_k\circ f_k^{\circ n}\circ L_k^{-1}$ converges to
  $\psi^{\circ n}\in \mathrm{Rat}_{d^n}$. From the previous paragraph, we conclude that $h$ is the unique hole of $g_n$, for all $n \ge 1$. Moreover,
  $d_h(g_n)=d^n$ because $\deg \tilde{g}_n=0$. Hence
   $(g_n)^* \mu /d^n = \delta_h$, for all  probability measures $\mu$ that are non-exceptional for $(g_n)$.
\end{proof}

\subsection{Degeneration with large degree growth}
\label{s:ramified}
The aim of this section is to apply the Structure Theorem~\ref{p:structure} and its corollaries, to prove the following:

\begin{proposition}
  \label{p:ramified}
  Let $\{f_k\}$ be a sequence in $\Ratd$ totally converging to
  $(g_n,\varphi_n)$ such that $[f_k]$ converges to $[g] \in \partial \ratd$.
  Assume that 
  $\deg \tilde{\varphi}_{n} \neq o(d^n)$.
  Then there exists a probability measure $\nu$ such that
   $$\nu  = \lim_{k\to\infty} \mu_{f_k} = \lim_{n\to\infty} \dfrac{1}{d^n} (g_n)^* \mu ,$$
 for any probability measure $\mu$ that is non-exceptional for $(g_n)$.
\end{proposition}

\begin{proof}
  Consider an accumulation point $\nu$ of $\{\mu_{f_k}\}$.
  After passing to a subsequence, we may assume  $\mu_{f_k} \to \nu$.
  We may also assume that there exists a collection of scalings $\{A_{n,k} : n,k \ge 1\}$ for $\{f_k\}$ with convergent changes of coordinates.
  Since $\deg \tvarphi_n \neq o(d^n)$, any sufficiently large $n$ is a fully ramified time. Thus, we may apply
  Corollary \ref{coro:main} (2) to conclude that
  $(A_{n,k})_*\mu_{f_k} \to \delta_{a_n}$,  for  all sufficiently large $n$.
  By Proposition \ref{prop:mea-CE}, we have that 
  $$\nu=\frac{1}{d^n}(\varphi_n)^\ast\delta_{a_n},$$
and, therefore,
  \begin{eqnarray*}
\nu&=&\frac{1}{d^n}\left(\sum_{h\in\mathrm{Hole}(\varphi_n)}d_h (\varphi_n) \cdot \delta_h+(\tilde{\varphi}_n)^\ast\delta_{a_n}\right)\\
&=&\frac{1}{d^n}\left(\sum_{h\in\mathrm{Hole}(\varphi_n)}d_h (\varphi_n) \cdot \delta_h+\sum_{\tilde{\varphi}_n(h)=a_n}\deg_h \tilde{\varphi}_n \cdot \delta_{h}\right).   
  \end{eqnarray*}
By Lemma \ref{l:constant}, we may assume that $g_n$ has constant reduction and
apply Lemma \ref{lem:s} to conclude that $\hole(g_n) = \hole(\varphi_n) \cup
\tilde{\varphi}^{-1}_n(a_n)$. Then
$$\nu=\frac{1}{d^n}\sum_{h\in\mathrm{Hole}(g_n)} d_h(g_n) \cdot \delta_h = \dfrac{1}{d^n} g_n^\ast \mu,$$
for any probability measure $\mu$ that is non-exceptional for $g_n$ and,
as a consequence,  $\mu_{f_k}$ converges to $\nu$.
\end{proof}

\subsection{Proof of Theorem \ref{A}}
\label{s:A}
Given any sequence $\{f_k\} \subset \Ratd$
converging to $\mathbf{g} = (g_n) \in \wRatd$,
we simultaneously show that there exists a probability measure $\mu_{\mathbf{g}}$ such that 
$$\dfrac{1}{d^n}(g_n)^* \mu \to \mu_{\mathbf{g}},$$
and that if along a subsequence of $\{f_k \}$ we have convergence
of the measures $\mu_{f_k}$ to a measure $\nu$, then $\nu = \mu_{\mathbf{g}}$.

So given a subsequence such that the measures of maximal entropy converge, 
there exists a further subsequence
$\{f_{k_i}\}$ such that $\{f_{k_i}\}$ totally converges
to $(g_n,\varphi_n)$ and $[f_{k_i}] \to [g] \in \oratd$.
If $\deg \tilde{\varphi}_n = o(d^n)$ or $[g] \in \ratd$, by Propositions \ref{p:smalldegree} and \ref{p:potential},
there exists a probability measure $\mu_{\mathbf{g}}$ such that 
$$\dfrac{1}{d^n}(g_n)^* \mu \to \mu_{\mathbf{g}}=\nu.$$ 
If $[g] \in \partial \ratd$ and $\deg \tilde{\varphi}_n \neq o(d^n)$,
then we apply Proposition \ref{p:ramified} to obtain the same conclusion.
\hfill $\Box$

\section{Continuous extension of barycentered measures}
\label{s:moduli}

The moduli space $\ratd$ is a complex orbifold of dimension $2d- 2$. Its compactification $\oratd$ is the categorical quotient of semistable points in $\oRatd$ by the $\mathrm{SL}(2,\C)$ conjugation action, see \cite{Silverman98}. According to \cite[Section 3]{DeMarco07}, the (semi)stable points can be characterized by the depths; more precisely, 
\begin{itemize}
\item a point $\varphi\in\oRatd$ is \emph{semistable} if 
$d_z(\varphi)\le (d+1)/2$ for any $z\in\oC$ and $\tilde\varphi(h)\not=h$ for any $h\in\oC$ with $d_h(\varphi)\ge d/2$;
\item a point $\varphi\in\oRatd$ is \emph{stable} if 
$d_z(\varphi)\le d/2$ for any $z\in\oC$  and $\tilde\varphi(h)\not=h$ for any $h\in\oC$ with $d_h(\varphi)\ge (d-1)/2$. 
\end{itemize}

In \S \ref{s:bary}, we review basic facts from~\cite{DeMarco07}.
In particular, to  each element in $\ratd$ one may continuously assign a well-defined barycentered measure of maximal entropy, modulo  push-forward by a rotation in $\operatorname{SO}(3)$. In \S \ref{s:B}, we establish Theorem \ref{B}, which asserts continuous extension of this assignment to 
the resolution space $\wratd \subset \Pi_{n\ge1} \oratdn$ of the iterate map $\oratd\dasharrow\Pi_{n\ge 1}\oratdn$ discussed in the introduction~\S \ref{s:introduction}.

\subsection{Barycentered measures}\label{s:bary}

Recall that $M^1(\oC)$ is the space of probability measures endowed with the weak*-topology. Let $M^1_o(\oC)\subset M^1(\oC)$ be the subset of probability measures with no atom of weight $\ge 1/2$. Then, under the identifications of $\oC$ with $S^2\subset\R^3$ by the stereographic projection and of hyperbolic space $\H^3$ with the unit ball  in $\R^3$, the \emph{conformal barycenter $C(\mu) \in \H^3$} of $\mu\in M^1_o(\oC)$ is well-defined  and continuous on $M^1_o(\oC)$. More precisely,  given $\mu\in M^1_o(\oC)$, the  barycenter $C(\mu)$ is uniquely determined
by the following two properties: 
\begin{enumerate}
\item  $C(\mu)=0\in\H^3$ if and only if $\int_{S^2}xd\mu(x)=0$; and  
\item $C(A_*\mu)=A(C(\mu))$ for any $A\in\Rat_1\cong\text{Isom}^+\H^3$.
\end{enumerate}
The conformal barycenter $C(\mu)$ of any $\mu\in M^1(\oC)\setminus M^1_o(\oC)$ is undefined. For more details, see \cite[Section 8]{DeMarco07} and \cite{Douady86}. 

A measure $\mu \in M^1(\oC)$ is \emph{barycentered} if $C(\mu)$ is well-defined and equals to $0$. Denote by $M_{bc}^1(\oC)\subset M^1(\oC)$ the subset of barycentered measures, and let $\overline{M}^1_{bc}(\oC)$ be the closure of $M_{bc}^1(\oC)$ in $M^1(\oC)$. Each element in $\overline{M}^1_{bc}(\oC)$ is either barycentered or of  the form $(\delta_a+\delta_{-1/\bar a})/2$.
for some $a\in\oC$, see \cite[Lemma 8.3]{DeMarco07}. Then both $M^1_{bc}(\oC)$  and  $\overline{M}^1_{bc}(\oC)$ are invariant under the push-forward  action of the compact group of rotations $\operatorname{SO}(3)\subset\sl2c$. Consider the quotient topological spaces 
$$M^1_{dm}(\oC):=M^1_{bc}(\oC)/\operatorname{SO}(3)\ \ \text{and}\ \ \overline{M}^1_{dm}(\oC):=\overline{M}^1_{bc}(\oC)/\operatorname{SO}(3).$$
The space $M^1_{dm}(\oC)$ is  locally compact, Hausdorff and its one-point compactification is $\overline{M}^1_{dm}(\oC)$, see \cite[Theorem 8.1]{DeMarco07}. Denote by $[\infty]$ the unique point in $\overline{M}^1_{dm}(\oC)\setminus M^1_{dm}(\oC)$. 

For each $\mu\in M^1_o(\oC)$, there exists $A\in\Rat_1$ such that $A_*\mu\in M^1_{bc}(\oC)$. 
Writing $[\mu]:=[A_*\mu]\in M^1_{dm}(\oC)$, one obtains a continuous map 
$$
\begin{array}{cccc}
  &M^1_o(\oC)&\rightarrow& M^1_{dm}(\oC)\\
    &    \mu & \mapsto & [\mu ], 
    \end{array}
$$
which induces a continuous map $\rat_d\to M^1_{dm}(\oC)$, sending $[f]$ to $[\mu_f]$. We remark here that the above continuous map can not extend continuously to $M^1(\oC)$; for instance, considering $\mu_n\in M^1(\oC)$ such that $\mu_n(\{0\})=\mu_n(\{1/n\})=\mu_n(\{1\})=1/3$, we have that $[\lim_{n\to\infty}\mu_n]=[\infty]\not=\lim_{n\to\infty}[\mu_n]$. 

\subsection{Proof of Theorem \ref{B}}\label{s:B}
Consider a sequence $\{[f_k]\}\subset\Ratd$
converging to $([g_n])$ in $\wratd$.
Let $\mu\in M^1(\oC)$ be a non-exceptional measure for $(g_n)$.
For notational simplicity, 
\begin{eqnarray*}
  \mu_k &:=& \mu_{f_k}\\
  \nu_n &:=& \dfrac{(g_n)^*\mu}{d^n}.
\end{eqnarray*}
Given accumulation points $[\mu_\infty]$ of $\{[\mu_k]\}$
and $[\nu]$ of $\{[\nu_n]\}$, after some work, we will show that
$[\nu]=[\mu_\infty]$, that is, Theorem \ref{B} holds.  By continuity
of the measure of maximal entropy for non-degenerate rational maps~\cite{Mane83}, {if $[g_1]\in\rat_d$, equivalently, $[f_k]$ converges in $\rat_d$, then $[\nu]=[\mu_\infty]$. Thus we work under the assumption that
  $$[g_1]\in\partial\rat_d.$$}

Passing to a subsequence in $k$, we assume that $[\mu_k]$ converges
to $[\mu_\infty]$. Passing to a subsequence $n_i$, we assume
that $[\nu_{n_i}]$ converges to $[\nu]$.
It is sufficient to
establish that $[\mu_\infty]$ and $[\nu]$ coincide
when $[\mu_\infty] \neq [\infty]$ or  $[\nu] \neq [\infty]$.

\begin{lemma}\label{lem:infinity}
Assume that $[\mu_\infty]\not=[\infty]$. Then 
$$[\nu]=\lim_{n\to\infty}[\nu_n]=[\mu_\infty].$$
\end{lemma}

\begin{proof}
After changing the representative of $[f_k]$, we may assume that $\mu_{f_k}$ is
  barycentered, for all $k$. Passing to a subsequence, $\{f_k\}$ converges
  in $\wRatd$, say to $\mathbf{h}=(h_n)$. By Theorem~\ref{A}, we have
  that $\mu_{f_k}$ converges to some measure
  $\mu_{\mathbf{h}}$. Since 
  $[\mu_\infty] \neq [\infty]$, we have $[\mu_{\mathbf{h}}]=[\mu_\infty] \neq [\infty]$.
  Therefore, $\mu_{\mathbf{h}}(z) < 1/2$ for all $z \in \oC$. Hence,
  for $n$ sufficiently large, $h_n$ is GIT-stable. Therefore, $[f_k^{\circ n}]$ converges to $[h_n]$ and $[h_n]=[g_n]$. In particular, $[(h_n)^* \mu/d^n]=
  [\nu_n]$ provided that $\mu$ is not exceptional for $h_n$, for all $n$.  The sequence $(h_n)$ is uniquely determined up to $\operatorname{SO}(3)$ conjugacy (the same for all $n$). For any given
  $n$, there are at most countably many elements of $\operatorname{SO}(3)$ so that
  $\mu$ is exceptional for the conjugacy of $h_n$ by that element. Hence
  we may find a conjugate so that $\mu$ is not exceptional for all $h_n$. Thus the conclusion holds. 
\end{proof}

Now we assume that $[\nu] \neq [\infty]$.
Passing to a further subsequences of $\{f_k\}$
and $\{g_{n_i}\}$ the following hold:
\begin{enumerate}
\item There exists $\varepsilon >0$ such that, for all $n_i$ and all $z\in\oC$,
  $$d_z(g_{n_i}) < \left(\dfrac{1}{2} -  2 \varepsilon\right) \cdot d^{n_i}.$$
\item For each $n_i$, there exists $B_{n_i,k} \in \Rat_1$ such that,
  as $k \to \infty$, 
  $$B_{n_i,k} \circ f_k^{\circ n_i} \circ B_{n_i,k}^{-1} \to g_{n_i}.$$
\item Let $$f_{n_i,k}:=B_{n_i,k} \circ f_k  \circ B_{n_i,k}^{-1}.$$
  Then, for all $n_i$ and $m$ there exists $g_{n_i,m} \in \oRat_{d^{m}}$
  such that
  $$f_{n_i,k}^{\circ m} \to g_{n_i,m}.$$
\item For each $n_i$ and each $m\ge 1$, there exist a collection of scalings $\{A_{n_i,m,k}\}$
  and maps $\psi_{n_i,m} \in \Rat^*_{d^{m}}$ 
  such that, as $k\to\infty$, 
  $$A_{n_i,m,k} \circ f_{n_i,k}^{\circ m} \to \psi_{n_i,m}.$$
\item There exists  $c \in [0,1]$ such that, as $n_i \to \infty$, 
  $$\dfrac{\deg \tilde\psi_{n_i,n_i}}{d^{n_i}} \to c \in [0,1].$$ 
\end{enumerate}
Note that $g_{n_i}$ is GIT-stable for each  $n_i$ by statement (1).
Thus we may choose $B_{n_i, k}$ as in statement (2).
Statements (3) and (4) say that $\{f_{n_i,k}\}$ totally converges to $(g_{n_i,m},\psi_{n_i,m})$ with associated scalings $\{A_{n_i,m,k}\}$.

Observe that
$g_{n_i,n_i} = g_{n_i}$. If $\tilde{g}_{n_i,m}$ is non-constant, then 
in (4) we choose the scaling $A_{n_i,m,k}$ to be the identity and, therefore,
$\psi_{n_i,m}={g}_{n_i,m}$.

For brevity, write $\psi_{n_i}:=\psi_{n_i,n_i}$.
 We establish the following two lemmas according to
the growth of 
$\deg\tilde\psi_{n_i}$; more precisely, according to whether $c>0$ or $c =0$.

\begin{lemma}\label{lem:inftyO}
  Suppose that $c>0$, that is, $\deg\tilde\psi_{n_i}\not=o(d^{n_i})$. Given a sufficiently large $n_i$,  
for any $\ell\ge n_i$ and any $z\in\oC$, 
$$\frac{d_z(g_{n_i,\ell})}{d^\ell}=\frac{d_z(g_{n_i})}{d^{n_i}}.$$
Moreover,
$$[\nu] = [\nu_{n_i}] = [\mu_\infty].$$
\end{lemma}
\begin{proof}
  Since  $c>0$, for sufficiently large $n_i$,
  the number of times $m<n_i$ such that $\deg\tilde\psi_{n_i,m+1} <
  d \cdot \deg \tilde\psi_{n_i,m}$ is bounded independently of $n_i$.
  Thus, 
   we can pick $N$ sufficiently large so that for any $n_i\ge N$,
  there is $1\le m_i\le n_i-5$ satisfying that $A_{n_i,m_i+4,k}\circ f_{n_i,k}^{\circ 5}\circ A_{n_i,m_i,k}^{-1}$ converges in $\Rat_{d^5}$. 
  {Since we are working under the assumption that $[g_1]\in\partial\rat_d$, } by Corollary \ref{coro:main1}, we conclude  that for any $\ell\ge n_i$ and any $z\in\oC$, the reduction $\tilde g_{n_i,\ell}$ is a constant independent of $\ell$ and  
$$\frac{d_z(g_{n_i,\ell})}{d^\ell}=\frac{d_z(g_{n_i})}{d^{n_i}}.$$ 
Hence $$\frac{(g_{n_i,\ell})^*\mu}{d^\ell}=\frac{(g_{n_i})^*\mu}{d^{n_i}}=\nu_{n_i}.$$ 
Since  by Theorem \ref{A} 
$$\lim_{\ell\to\infty}\frac{(g_{n_i,\ell})^*\mu}{d^\ell}=\lim\limits_{k\to\infty}\mu_{f_{n_i,k}}, $$
we have that 
$$\nu_{n_i}=\lim_{k\to\infty}\mu_{f_{n_i,k}}.$$
{Observing that $\nu_{n_i}\in M^1_o(\oC)$ since $g_{n_i}$ is stable, by the continuity of the map $M^1_o(\oC)\rightarrow M^1_{dm}(\oC)$, we obtain that 
$$[\nu_{n_i}]=\lim\limits_{k\to\infty}[\mu_{f_{n_i,k}}]=\lim_{n\to\infty}[\mu_{f_k}]=[\mu_\infty].$$
Thus the conclusion follows. }
\end{proof}

\begin{lemma}\label{lem:inftyo}
  Suppose that $c=0$, that is, $\deg\tilde\psi_{n_i}=o(d^{n_i})$. Then
 $$[\nu]=[\mu_\infty]$$
\end{lemma}
\begin{proof}
  By Lemma~\ref{lem:infinity}, we may assume that $[\mu_\infty]=[\infty]$.
  Passing to a subsequence in $k$ we can assume that $B_{n_i,k}\circ B_{n_j,k}^{-1}$ converges to some $B_{n_j,n_i}\in\oRat_1$ for any $n_i\not=n_j$.  Then passing to a further subsequence in $n_i$ if necessary, we can assume that $B_{n_j,n_i}\in\Rat_1$ for all $n_i,n_j$ or $B_{n_j,n_i}\in\partial\Rat_1$ for all $n_i,n_j$.
 
\noindent
  \emph{Case 1:}  $B_{n_j,n_i}\in\Rat_1$ for all $n_i\not=n_j$.  
  Since $\{f_{n_1,k}\}$ totally converges
  to $(g_{n_1,m}, \psi_{n_1,m})$, passing to a subsequence in $k$,  
  we may choose $\{A_{n_1,n_i,k}\}$ so that
  $\psi_{n_i} = \psi_{n_1,n_i} \circ C_{n_i}$ for some $C_{n_i} \in \Rat_1$.

   By Lemma \ref{lem:depthmeasure},   let
   $$\nu':=\lim_{i\to\infty}{\eta_{\psi_{n_1,n_i}}} \in M^1(\oC),$$
   where $\eta_{\psi_{n_1,n_i}}$ is the corresponding depth measure
   from Definition \ref{def:depthmeasure}.
   We claim that $\nu'\in M^1_o(\oC)$ and hence $[\nu']\not=[\infty]$.
Indeed, given $z \in \oC$, let $w_i = C_i^{-1}(z)$, then
$$\nu'(\{z\}) =\lim \dfrac{d_z (\psi_{n_1,n_i})}{d^{n_i}} = \lim \dfrac{d_{w_i} (\psi_{n_i})}{d^{n_i}} \le \lim \dfrac{d_{w_i}(g_{n_i})}{d^{n_i}} \le \dfrac12 - 2 \varepsilon.$$
The first equality follows from the proof of Lemma~\ref{lem:depthmeasure} since
$\nu'(\{z\})$ is the limit of the non-decreasing sequence
$\{\eta_{\psi_{n_1,n_i}}(\{z\})=d^{-{n_i}} \cdot d_z (\psi_{n_1,n_i})\}$.
The second equality follows from  $\psi_{n_i} = \psi_{n_1,n_i} \circ C_{n_i}$.
The first inequality is Lemma \ref{lem:s} applied
to $z = w_i$ and $\{f_{n_i,k}^{\circ n_i}\}$.
The second inequality is statement (1). 
  
{Moreover, since $\deg\tilde\psi_{n_i}=o(d^{n_i})$, we have that $\deg \tilde{\psi}_{n_1, n_i} = o (d^{n_i})$, as $i\to\infty$, and hence $\deg \tilde{\psi}_{n_1,n} = o(d^n)$, as $n\to\infty$. Thus, applying Proposition \ref{p:smalldegree} to $\{f_{n_1,k}\}$, we conclude that 
  $$\lim_{k\to\infty} \mu_{f_{n_1,k}} = \lim_{i\to\infty}\frac{(g_{n_1,n_i})^* \mu'}{d^{n_i} } =\nu',$$
  whenever $\mu'\in M^1(\oC)$ is any non-exceptional measure for $(g_{n_1,n})$.
Continuity of the map $M^1_o(\oC)\rightarrow M^1_{dm}(\oC)$ yields
$$[\mu_\infty]= \lim_{n\to\infty}[\mu_{f_k}]=\lim_{k\to\infty}[\mu_{f_{n_1,k}}]=[\nu'].$$
   Hence $[\mu_\infty]\neq [\infty]$, which is a contradiction.}

  \smallskip
  \noindent
  \emph{Case 2:} $B_{n_j,n_i}\in\partial\Rat_1$ for any $n_i\not=n_j$.
  We show that this case cannot occur under our assumptions; more precisely, by counting preimages, we will obtain a contradiction with statement (1).

 Since $\deg\tilde\psi_{n_i}=o(d^{n_i})$, we consider $n_{i_0}$ such that for $n_i \ge n_{i_0}$,
 $$\deg\tilde\psi_{n_i}<\varepsilon\cdot d^{n_i}.$$
 Let $\ell:=n_{i_0}$ and $s:=n_{i_0+1}$.
 For brevity write
 \begin{equation*}
    \gamma_k  :=  B_{\ell,k} \circ B_{s,k}^{-1}
  \end{equation*}
  with $\gamma_k \to B_{(s,\ell)}$ and $\gamma_k^{-1} \to B_{(\ell,s)}$ in $\oRat_1$.
 Let $b_{\ell}:=\tilde B_{s,\ell}$ and
 $b_{s}:=\tilde B_{\ell,s}$.

 Intuitively, our argument is as follows.
 We restrict our attention to the collection of scalings $\{B_{\ell,k},
 B_{s,k} \}$ with associated space formed by two spheres $\oC_\ell$ and
 $\oC_s$. 
 We consider the set of preimages $S$ of a generic point in $\oC_s$, under
 an appropriate iterate  $f_{s,k}^{\circ s}$. Then statement (1) implies
 that less than half of the elements of $S$
 map, under $f_{s,k}^{\circ s}\circ\gamma_k^{-1}$, close to  $b_\ell$, the hole of $B_{(\ell,s)}$.
 So more than half of the elements are bounded away from $b_\ell$.
 Applying the change of coordinates $\gamma_k^{-1}$, 
 one concludes that more than half of the elements
 of $S$  are
 close to $b_s$, which is a contradiction with statement (1).

More precisely, consider 
\begin{eqnarray*}
  \alpha_k &:= & B_{s,k} \circ f_k^{\circ s-\ell} \circ B_{\ell,k}^{-1},\\
  \beta_k &:= & B_{\ell,k} \circ f_k^{\circ \ell} \circ B_{\ell,k}^{-1}.
  \end{eqnarray*}
  Note that
  $$f_{s,k}^{\circ s} =\alpha_k \circ \beta_k \circ \gamma_k \to g_s.$$
  Let $V$ be a small neighborhood of $b_{s}$.
  For $w\in\oC$, we now estimate $\#(f_{s,k}^{-s}(w)\cap V)$. 
First observe that $\#(\alpha_k)^{-1}(w)=d^{s-\ell}$. Given a small neighborhood $D$ of $b_{\ell}$, since $d_{b_{\ell}}(g_\ell) < (1/2-2\varepsilon)\cdot d^\ell$ from the statement (1) in the assumptions on $\{f_k\}$ and $\{g_{n_i}\}$, by Lemma \ref{lem:pre-number} (2), we have that 
$$\#((\alpha_k\circ\beta_k)^{-1}(w)\cap D)< (1/2-2\varepsilon)\cdot d^s+\varepsilon\cdot d^s=(1/2-\varepsilon)\cdot d^s.$$
Therefore,
$$\#((\alpha_k\circ\beta_k)^{-1}(w)\cap(\oC\setminus D))\ge(1/2+\epsilon)\cdot d^s.$$
In  $ \oC \setminus D$,  we have uniform convergence of $\gamma^{-1}_k$
to the constant $b_s$. 
We conclude that 
$$\#(f_{s,k}^{\circ -s}(w)\cap V)\ge(1/2+\varepsilon)\cdot d^s.$$
By Lemma \ref{lem:pre-number} (2), we have that
$d_{b_{s}}(g_s)\ge d^s/2$, which contradicts statement (1). 
\end{proof}

This completes the proof of Theorem \ref{B}. \qed

\appendix
\section{Trees of Bouquets of Spheres}
\label{s:appendix}

\subsection{Hausdorff limits of spheres}
Let $J$ be a finite set of indices. Consider a collection of scalings
$$\mathbf{A}:=\{ A_{n,k}: n \in J,k \ge 1 \} \subset \Rat_1,$$
  indexed by $J$.
Assume that for all $n,n' \in J$ with $n \neq n'$,
there exists $A_{(n,n')} \in \partial \Rat_1$ such that
$$A_{n',k} \circ A_{n,k}^{-1} \to A_{(n,n')}\ \ \text{as}\ \  k \to \infty.$$
We say that $\mathbf{A}$ is \emph{a collection of independent scalings indexed by $J$.}


For $ n \neq n'$, let
$$a_{n,n'} := \widetilde{A}_{(n,n')}.$$
Note that
$$a_{n',n} = \operatorname{Hole}(A_{(n,n')}).$$
For each $n \in J$, consider a copy of $\oC$ inside $\oC^{J}$:
$$\oC_n:= \left\{ (w_j) \in  \oC^{J} : w_j = a_{n,j}
\text{ for all } j \neq n \right\}.$$
Denote by $\rho_n:\oC^{J} \to \oC$ the projection to the $n$-th coordinate.
This projection  provides a \emph{standard coordinate  for $\oC_n$}. That is,
$\rho_n : \oC_n \to \oC$ is  a conformal isomorphism.

Consider the union of Riemann spheres 
 $$\cS := \bigcup_{n\in J} \oC_n \subset  \oC^{J}.$$
 We will show that $\cS$ is a tree of bouquets of spheres
 that arises as the
 Hausdorff limit of Riemann spheres. 
 We say that \emph{$\cS$ is the  space  associated to
   $\mathbf{A}$.}

 \smallskip
For all $n \neq n' \in J$,
note that $\rho_n (\oC_{n'}) = \{a_{n',n}\}$.
Moreover, the intersection $\oC_{n'} \cap \oC_{n}$
is not empty if and only if 
$$a_{n',n''} = a_{n,n''}$$
for all $ J \ni n'' \neq n,n'$.
In this case, the intersection point is $a_{n',n} \in \oC_{n}$
and $a_{n,n'} \in \oC_{n'}$ in the respective standard coordinates.

An alternative description of $\cS$ can be given
as the quotient of the disjoint union of abstract Riemann spheres
$\sqcup_{n \in J}\oC_{n}$ by an equivalence relation $\sim$.
In fact, let $\sim$ be the equivalence relation such that $a_{n',n} \in \oC_{n}$ is identified with $a_{n,n'} \in \oC_{n'}$ if and only
  if $$a_{n,n''} = a_{n',n''}$$
  for all $ n'' \neq n, n'$. All other $\sim$-classes are trivial.
  Then
  $$\cS \equiv \sqcup_{n \in J}\oC_{n}/ \sim.$$

  Associated to $\mathbf{A}$ we also have a sequence of embedding of
  $\oC$ in $\oC^{J}$:
$$
\begin{array}{cccc}
  \cA_{k}:&\oC&\to& \oC^{J}\\
          & z & \mapsto & (A_{n,k}(z))_{n \in J}.
\end{array}
$$

  \begin{proposition}
    \label{p:bouquet}
    Consider a collection $\mathbf{A}$
     of independent scalings indexed by
    $J$. Let $\cS \subset \oC$
    be the space associated to $\mathbf{A}$. If
    $\cA_k : \oC \to \oC^J$ is the associated sequence of embeddings,
  then
  $$\cS = \cap_{k_0 \ge 1} \overline{\cup_{k \ge k_0} \cA_k (\oC)}.$$
  Moreover, $\cS$ is simply connected.
\end{proposition}

\begin{proof}
    Consider points $(w_j)\in \cS$ such that $w_{\ell} \neq a_{j,\ell}$ and
  $w_j=a_{\ell,j}$ for all $j \neq \ell$. Such elements are dense
  in $\oC_\ell$. Denoting by $z_k := A^{-1}_{\ell,k}(w_\ell)$, we have that $\cA_{k}(z_k) \to (w_j)$.
  So $\cS$ is contained in the Hausdorff limit of $\{\cA_{k}(\oC)\}$.

  We prove that the Hausdorff limit of $\{\cA_ k(\oC)\}$ is contained in
  $\cS$. 
  Suppose on the contrary that $\{z_k\}$ is a sequence such that $\cA_ k(z_k)$ converges to
  some $(u_j) \notin \cS$. Then for all $j\in J$, we may choose
  $J \ni s(j) \neq j$ such that $u_{s(j)} \neq a_{j,s(j)}$.
  Therefore,  $u_j = a_{s(j),j}$ since
  the change of coordinates 
  $A_{j,k} \circ A_{s(j),k}^{-1}$ converges uniformly in
  compact subsets of  $\oC\setminus  \{a_{j,s(j)}\}$
  to $a_{s(j),j}$.  Now consider $j_0\in J$ and let $j_i =s^{\circ i}(j_0)$ for $i\ge 1$. We may assume that
  $j_p=j_0$ for a minimal $p \ge 2$. 
  Since $a_{j_0,j_1} \neq u_{j_1} = a_{j_2,j_1}$, it follows that
  $a_{j_2, j_0} = a_{j_1,j_0}$. Recursively, we obtain that $a_{j_{p-1},j_0}=a_{j_1,j_0}=u_{j_0}$.  However, $u_{j_0}=u_{s(j_{p-1})} \neq a_{j_{p-1},j_0}$. It is a contradiction.

  Any Hausdorff limit of connected sets is connected, so $\cS$
  is connected. We now show that $\cS$ is simply connected. 
  By contradiction, suppose that for some $q \ge 2$,
there exists a cycle of spheres of length $q$.
That is, there exist a chain of pairwise
distinct spheres $\oC_{\ell_0},\dots,\oC_{\ell_{q-1}}$
that intersect at pairwise distinct points $x_0, \cdots, x_{q-1}$;
that is, $\{ x_i \}= \oC_{\ell_i} \cap \oC_{\ell_{i+1}} $ for all $0 \le i <q$,
subscripts mod $q$.
Then
$x_{j+1}-x_j$
is a vector in $\C^J$ with unique non-vanishing
component in its $\ell_j$-coordinate for all $j$. But
$\sum x_{j+1}-x_j =x_0-x_0=0$,
 which is a contradiction.
\end{proof}

\subsection{Maps between trees of bouquets of spheres}
  Let $J, J'$ be finite sets.
  Consider collections
  $$\mathbf{A}:=\{A_{n,k} : n \in J, k \ge 1\} \subset \Rat_1,$$
  $$\mathbf{B}:=\{B_{n,k} : n \in J', k \ge 1\} \subset \Rat_1.$$
  Assume that $\mathbf{A}$ and $\mathbf{B}$ are collections of
  independent scalings.

  Consider a sequence $\{f_k\} \subset \Ratd$. We say that $\{f_k\}$
  \emph{maps $\mathbf{A}$ into $\mathbf{B}$ by $\tau: J \to J'$} if for all $n \in J$,
$$B_{\tau(n),k} \circ f_k \circ A^{-1}_{n,k}$$
converges to some element $\varphi_{n,\tau(n)}\in\Rat^*_d$.
When $\varphi_{n,\tau(n)}\in\Rat_d$, for all $n \in J$, we say that
$\{f_k\}$ is \emph{fully ramified} at $\bA$.

In coordinates, given by the associated embeddings $\cA_k$ and $\cB_k$
of $\oC$ into $\oC^J$ and $\oC^{J'}$,
the maps $f_k$ become 
 $$
  \begin{array}{cccc}
    F_k :& \cA_k(\oC)& \to &\cB_k(\oC)\\
         &    w & \mapsto & \cB_k \circ f_k \circ \cA_{k}^{-1} (w).
  \end{array}
  $$
  In the limit, $F_k$ is, in a certain sense, semiconjugate to
  the action of ${\varphi}_{n,\tau(n)}$ on $\oC_n$.
  More precisely, provided that $z$ is not a hole of ${\varphi}_{n,\tau(n)}$,
  $$\rho_{\tau (n)} \circ F_k (z) \to \tilde{\varphi}_{n,\tau(n)} (\rho_n(z)).$$
  In general, we have the following:


\begin{proposition}
  \label{p:not-fully}
  Consider a sequence $\{f_k\} \subset \Ratd$ and collections of independent
  scalings $\mathbf{A}$ and $\mathbf{B}$ indexed by $J$ and $J'$ with associated spaces $\cS$ and $\cS'$, respectively, such that
  $\{f_k\}$ maps $\mathbf{A}$ to $\mathbf{B}$ by $\tau:J \to J'$. Let
  $\cA_k$ and $\cB_k$ be the associated embeddings of $\oC$ into
  $\oC^{J}$ and $\oC^{J'}$, respectively.
  Let
  $$
  \begin{array}{cccc}
    F_k :& \cA_k(\oC)& \to &\cB_k(\oC)\\
         &    w & \mapsto & \cB_k \circ f_k \circ \cA_{k}^{-1} (w).
  \end{array}
  $$
  Consider $z_0 \in \oC_j \subset \cS$ and $w_0 \in \oC_{\tau(j)} \subset \cS'$.
  Then, for every small neighborhood $U_0 \subset \oC^{J} $
  of $z_0$,
  there exists a neighborhood $W_0 \subset \oC^{J'}$ of $w_0$ such that 
  for all $w \in \cB_k(\oC) \cap W_0$, 
  $$\#F_k^{-1} (w) \cap U_0 = \begin{cases}
    d_{z_0}(\varphi_{j,\tau(j)}), & \text{ if } \tilde{\varphi}_{j,\tau(j)}(z_0) \neq w_0,\\
    d_{z_0}(\varphi_{j,\tau(j)}) + \deg_{z_0}\tilde{\varphi}_{j,\tau(j)}, & \text{ if } \tilde{\varphi}_{j,\tau(j)}(z_0) = w_0.
  \end{cases}
$$
\end{proposition}

\begin{proof}
  For simplicity, let $S_k := \cA_k (\oC)$ and $S_k' := \cB_k (\oC)$.
  Note that $S_k \cap \rho_j^{-1}(U) = \cA_k \circ A^{-1}_{j,k}(U)$
  for any $U \subset \oC$. A similar statement holds of $S_k'$.

  Given a small neighborhood $U_0$ of $z_0$, by compactness, it
  contains $\rho_j^{-1}(U_0')$ for some neighborhood $U_0'$ of $\rho_j(z_0)$ in $\oC$.
  Let $h_k := A_{\tau(j),k} \circ f_k \circ A_{j,k}^{-1}$. Then
  $h_k \to \varphi_{j,\tau(j)}$. By Lemma \ref{lem:pre-number} (1), there exists
  a neighborhood $W_0'$ of $w'_0:=\rho_j(w_0)$ in $\oC$ such that 
  $$\# h_k^{-1} (w') \cap U_0' = d_{z_0} (\varphi_{j,\tau(j)}) + \deg_{z_0}\varphi_{j,\tau(j)},$$
  for all $w' \in W_0'$. Let $W_0: = \rho^{-1}_{\tau(j)} (W_0')$.
  Given $w \in S_k \cap W_0$, consider $w' \in W_0'$ such that
  $\cB_k \circ A^{-1}_{\tau(j),k}(w') = w$. Now, $z' \in U'_0$ is such that
  $h_k(z') = w'$ if and only if $z=\cA_k \circ A^{-1}_{j,k}(z')$ is such that
  $F_k (z) =w$. That is, $\cA_k \circ A^{-1}_{j,k}$ is a bijection between
  $F_k$-preimages of $w$ in $U_0$ and $h_k$-preimages of $w'$ in $U_0'$.
\end{proof}

\subsection{Fully ramified maps}
\label{s:frmaps}
Let us now consider the special case in which
the limiting degree at each sphere is maximal.

\begin{proposition}
  \label{p:critical}
  Consider a sequence $\{f_k\} \subset \Ratd$
  that maps $\bA$
  to $\bB$ by a bijection $\tau: J \to J'$,
  where  $\bA$ and $\bB$ are independent collections of scalings 
  with associated spaces $\cS$ and $\cS'$.
  If $\{f_k\}$ is fully ramified at $\bA$, then
  there is a well defined and continuous map $\cF: \cS \to \cS'$
given by:
  \begin{equation*}
  \label{eq:defcF}
  \cF(z) :=\varphi_{n,\tau(n)}(z) \in \oC_{\tau(n)},
\end{equation*}
for $z \in \oC_n$ and $n \in J$.
  Moreover, there exist a finite set $\crit(\cF) \subset \cS$
  and a (multiplicity) function $m: \crit(\cF) \to \N$ such that the following hold: 
  \begin{enumerate}
  \item $$\sum_{\omega \in \crit(\cF)} m(\omega) = 2d-2.$$
  \item The multiplicity of  a critical point $x$ of $\varphi_{n,\tau(n)}$
    is $$\sum_{\omega \in \rho^{-1}_n(x) \cap \crit(\cF)} m(\omega).$$
  \end{enumerate}
\end{proposition}

We say that $\omega \in \crit \cF$ is a \emph{critical point of multiplicity} $m(\omega)$ of $\cF$. 

\begin{proof}
Recall that to obtain $\cS$,  we identify $a_{n,n'}$ with $a_{n',n}$
if and only if 
$$a_{n',n''} = a_{n,n''}$$
for all $ J \ni n'' \neq n,n'$. Similarly, we identify points of different spheres in $\cS'$.
Since $\tau: J \to J'$ is a bijection, for the continuity of $\cF$,
it suffices to check that $\varphi_{n,\tau(n)}(a_{n',n}) = a_{\tau(n'),\tau(n)}$ for all $n \neq n' \in J$.
Choose $z \in \oC_{n'}$ such that
$z \neq a_{n,n'}$ and $\varphi_{n',\tau(n')} (z) \neq a_{\tau(n),\tau(n')}$.
Then, uniformly in a neighborhood of $z$, the change of coordinates
$A_{n,k} \circ A^{-1}_{n',k}$ converges to $a_{n',n}$.
Thus $$(A_{\tau(n),k} \circ f_k \circ A_{n,k}^{-1} ) \circ (A_{n,k} \circ
A^{-1}_{n',k}) (z) \to \varphi_{n,\tau(n)}(a_{n',n}).$$
Reorganizing the above expression yields 
$$(A_{\tau(n),k} \circ A^{-1}_{\tau(n'),k} )\circ (A_{\tau(n'),k} \circ f_k \circ
A_{n',k}^{-1} ) (z) \to a_{\tau(n'),\tau(n)}.$$
Hence $\varphi_{n,\tau(n)}(a_{n',n}) = a_{\tau(n'),\tau(n)}$, as claimed.

Passing to a subsequence, we may label the critical points of $f_k$
by $c_k(1), \dots, c_k(2d-2)$ with repetitions according to
multiplicity. We may also assume that $\cA_k (c_k(i))$ converges to $c(i) \in \cS_{N_0,N}$, as $k\to\infty$, 
for all $i$. We say that $\omega \in \cS_N$ is a \emph{critical point of $\cF$ of multiplicity
  $m$} if there exist exactly $m$ indices $i$ such that $c(i) =\omega$.
Let $\crit(\cF)$ be the set formed by the critical points. Note that, a priori,
$\crit (\cF)$ depends on the choice of the subsequence. It can be shown that
it is independent, but we will not use nor prove this fact.

Given $x \in \oC$,  
since $A_{n+1,k} \circ f_k \circ
A_{n,k}^{-1}$ converges locally uniformly to $\tilde\varphi_{n,n+1}$ away from the holes of $\varphi_{n,n+1}$, the number of critical
points of $A_{n+1,k} \circ f_k \circ
A_{n,k}^{-1}$ in a neighborhood $W$ of $x$, counted with multiplicity, is
$m: = \deg_x\tilde\varphi_{n, n+1} -1$. That is, the number of critical point of
$f_k$ in $A_{n,k}^{-1}(W)$ is $m$.
Therefore, $\cA_{k} (A_{n,k}^{-1}(W))$ contains $m$ critical points of $\cF$
for all $k$ sufficiently large. That is, $\rho_n^{-1}(x)$ contains
$m$ critical points $\omega \in \crit \cF$, counted with multiplicities.
\end{proof}

\bibliographystyle{alpha}
\bibliography{references}
\end{document}